\documentclass[reqno,a4paper,12pt]{amsart}
\usepackage{amsmath,amssymb,amsthm,amscd,amsfonts,amsbsy,bm}
\usepackage[mathscr]{eucal}

\usepackage[dvips]{color}
\usepackage{color}

\usepackage[active]{srcltx}

\font\sc=rsfs10 at 12pt

\numberwithin{equation}{section}

\renewcommand{\a}{\alpha}
\renewcommand{\b}{\beta}
\newcommand{\g}{\gamma}

\renewcommand{\d}{\delta}
\newcommand{\D}{\Delta}
\newcommand{\e}{\epsilon}
\newcommand{\ve}{\varepsilon}
\newcommand{\z}{\zeta}
\newcommand{\y}{\eta}

\renewcommand{\k}{\kappa}

\renewcommand{\l}{\lambda}

\newcommand{\m}{\mu}

\newcommand{\x}{\xi}

\renewcommand{\r}{\rho}
\newcommand{\s}{\sigma}

\newcommand{\vs}{\varsigma}

\newcommand{\f}{\phi}
\newcommand{\vf}{\varphi}

\renewcommand{\o}{\omega}

\newcommand{\C}{{\mathbb C}}
\newcommand{\R}{{\mathbb R}}
\newcommand{\Z}{{\mathbb Z}}

\newcommand{\DD}{{\mathbb{D}}}
\newcommand{\HH}{{\mathbb H}}

\newcommand{\kb}{{\mathbf k}}

\newcommand{\Bb}{{\mathbf B}}

\newcommand{\Fb}{{\mathbf F}}

\newcommand{\Ib}{{\mathbf I}}

\newcommand{\Pb}{{\mathbf P}}

\newcommand{\Sbb}{{\mathbf S}}
\newcommand{\Tb}{{\mathbf T}}

\newcommand{\HF}{\mathfrak H}

\newcommand{\SF}{\mathfrak S}

\newcommand{\Ac}{{\mathcal A}}
\newcommand{\Bc}{{\mathcal B}}

\newcommand{\Ec}{{\mathcal E}}
\newcommand{\Fc}{{\mathcal F}}
\newcommand{\Gc}{{\mathcal G}}
\newcommand{\Hc}{{\mathcal H}}

\newcommand{\Sc}{{\mathcal S}}

\newcommand{\Wc}{{\mathcal W}} 
\newcommand{\Xc}{{\mathcal X}}
\newcommand{\Yc}{{\mathcal Y}}

\newcommand{\Fs}{\sc\mbox{F}\hspace{1.0pt}}

\newcommand{\supp}{\hbox{{\rm supp}}\,}

\newcommand{\comp}{\footnotesize{\rm{Comp}}\,}

\DeclareMathOperator{\rank}{rank}

\newcommand{\loc}{\operatorname{loc}}

\newtheorem{theorem}{Theorem}[section]
\newtheorem{proposition}[theorem]{Proposition}
\newtheorem{lemma}[theorem]{Lemma}
\newtheorem{corollary}[theorem]{Corollary}

\newtheorem*{theorem*}{Theorem}

\theoremstyle{definition}
\newtheorem{definition}[theorem]{Definition}

\theoremstyle{remark}
\newtheorem{remark}[theorem]{Remark}
\newtheorem{example}[theorem]{Example}

\newcommand{\sasha}[1]{} 
\newcommand{\GR}[1]{}

\begin{document}

\title[]{Toeplitz operators in the Herglotz space}

\author[G. Rozenblum]{Grigori Rozenblum}
\address{1. Department of Mathematics \\
                          Chalmers University of Technology \\
                          2.Department of Mathematics  University of Gothenburg \\
                          Chalmers Tv\"argatan, 3, S-412 96
                           Gothenburg
                          Sweden}
\email{grigori@math.chalmers.se}
\author[Vasilevski]{Nikolai Vasilevski }
\address{Department of Mathematics, Cinvestav, Mexico-city, Mexico}
\email{nvasilev@duke.math.cinvestav.mx}
\begin{abstract}
We define and study Toeplitz operators in the space of Herglotz solutions of the Helmholtz equation in $\R^d$. As the most traditional definition of Toeplitz operators via Bergman-type projection is not available here, we use an approach based upon the reproducing kernel nature of the Herglotz space and  sesquilinear forms, which results  in a meaningful theory. For two important patterns of  sesquilinear forms we discuss a number of properties, including the  uniqueness of determining the symbols from the operator, the finite rank property, the conditions for boundedness and compactness, spectral properties, certain algebraic relations.
\end{abstract}
\date{}

\maketitle


\section{Introduction}\label{intro}
In this paper we address several questions concerning Toeplitz operators in the Herglotz (also-called Agmon-H\"ormander) space of solutions of the Helmholtz equation in $\R^d, d\ge2$. We discuss various definitions of the operators, their
representations, questions of uniqueness, finite rank and boundedness.

Toeplitz operators play an important role in Analysis, especially, in function theory and functional analysis. They appear also in the representation theory, quantization, and noncommutative geometry. In its turn,  problems arising in the theory of Toeplitz operators appeal to various fields of mathematics.

In the course of the development of the theory of Toeplitz operators, the point of view underwent  considerable  modification and generalization. The commonly used definition of the ``classical'' Bergman-type Toeplitz operators involves the following ingredients: a Hilbert space $L_2$ of functions defined in a domain in the (multidimensional) real or complex space; its closed subspace $\mathcal{B}$, as a rule consisting of  solutions of some elliptic equation or system; the orthogonal projection $P : L_2 \rightarrow \mathcal{B}$; and a bounded measurable function $a$.
Then the Toeplitz operator $\Tb_a$ with symbol $a$ acts as the compression of the multiplication (by~$a$) operator to the subspace $\mathcal{B}$:
\begin{equation}\label{0.1}
    \Tb_a u = P(au),\quad u \in \mathcal{B}.
\end{equation}
Numerous studies of various species of Toeplitz operators, defined by \eqref{0.1}, use quite different subspaces $\mathcal{B}$, including, in particular, the classical Hardy space $H^2$ being a closed subspace in the $L_2$ space on the unit circle; the (weighted) Bergman space consisting of square integrable analytic functions on the unit disk; the Fock (or Segal-Bargmann) space consisting of entire analytical functions square integrable against the Gaussian measure; harmonic and pluriharmonic Bergman spaces, poly-Bergman and poly-Fock spaces, etc.

This point of view prevailed for a considerable time, however, it turned out that certain important cases of operators, which were reasonable to consider as Toeplitz ones, do not fit the above scheme. The most outstanding here was the situation with singular symbols. In fact, formula \eqref{0.1} is quite meaningful as long as the function $a$ is bounded: for $u\in \mathcal{B} \subset L_2$, the function $au$ belongs to $L_2$ as well, and therefore we can legally apply the projection $P$ to it. However, if $a$ is unbounded, or even not a function (say, a measure or a distribution), then the `object' $au$ is not an element in $L_2$. Thus each such case requires additional specific analytic considerations.

An important common feature in all particular cases is that the subspaces where the Toeplitz operators act are the \emph{reproducing kernel spaces}.
This means that for any point $x$ in the domain in question the evaluation functional $\f_x : \, u \mapsto u(x)$ is continuous in the norm of $\mathcal{B}$. By the Riesz representation theorem, such functional must have a special form,
\begin{equation*} 
    \f_x(u)=\langle u, \k_x\rangle, \quad u\in \Bc,
\end{equation*}
where $\langle\cdot,\cdot\rangle$ is the scalar product in $\Bc$ and $\k_x(\cdot)$ is a  function in $\Bc$, depending on $x$. The function of two variables $\k_x(y)=\k(x,y)$ is called the reproducing kernel for $\Bc$.

A general unified approach to defining  Toeplitz operators, based upon the reproducing kernel subspaces, was developed in \cite{RozVas1}, \cite{RozVas2}. The role of $\Bc$ can be played by any reproducing kernel space while the  symbol is realized by a \emph{bounded sesquilinear form} $\Fb(u,v)$, $u,v\in\Bc$. As soon as such sesquilinear form is defined, the corresponding Toeplitz operator acts in $\Bc$ as
\begin{equation}\label{0.3}
    (\Tb_\Fb u)(x)=\Fb(u(\cdot),\k_x(\cdot)).
\end{equation}
Viability of such approach was demonstrated in   \cite{RozVas1}, \cite{RozVas2} for   Toeplitz operators acting on the classical Bergman space on the unit disk and the Fock space. As was shown there, definition \eqref{0.3} covers (apart of the standard classical case: bounded measurable function symbols) a large variety of singular symbols, such as unbounded functions,  measures, distributions, and even hyperfunctions in a certain class.

An important circumstance, to be emphasized here, is that definition \eqref{0.3} of a Toeplitz operator uses only a bounded sesquilinear form and a reproducing kernel, both of them being internal objects for the reproducing kernel space $\Bc$ in question. In this way, the enveloping Hilbert space, such as $L_2$, does not enter  the definition of a Toeplitz operator under this approach and thus can be completely ignored. Moreover, such an approach permits us to define and study Toeplitz operators in the situation when there are even no suitable candidates for the role of the enveloping space, and thus no corresponding orthogonal projection is available.

Considered in this paper Toeplitz operators on the Herglotz space are exactly of this nature. We define and study Toeplitz operators without any mentioning of an enveloping space (nevertheless some candidates to its role are discussed in Section \ref{Sect.envel}). The cornerstone of our study is definition \eqref{0.3} which uses the reproducing kernel structure of the Herglotz space as well as the explicit formula for the reproducing kernel, obtained in Section \ref{Sect.sesq}.

The general notion of Bergman type spaces involves their understanding as Hilbert spaces of solutions of some elliptic equation or system. In particular, the $\bar{\partial}$ operator leads to classical Bergman spaces of analytic functions, and the Fock type spaces of entire analytic functions in the complex plane or $\C^d$. The harmonic Bergman spaces as well as  poly-Bergman and poly-Fock spaces are the spaces of solutions of the Laplacian $\D$ and the iterated $\bar{\partial}$ operator. The next most important equation of interest, following the Cauchy-Riemann one and the Laplacian is, probably, the Helmholtz equation in $\R^d$, normalized as
\begin{equation}\label{0.Helm}
    \D u+u=0.
\end{equation}

Note that there is a certain possibility of choice, which of particular Hilbert spaces of solutions of \eqref{0.Helm} to select for study. It turns out that in a majority of applications the solutions $u$ satisfying the radiation conditions are needed for further analysis. Such solutions do not belong to $L_2(\R^d)$, but miss $L_2$ just a little bit (there are no nontrivial solutions of the Helmhotz equation in $L_2(\R^d)$). They  can be also described as Fourier transformed  $L_2$ functions on the unit sphere $\Sbb=S^{d-1}\subset \R^d$, considered as distributions in $\R^d$. Such solutions are usually called Herglotz solutions, and the norm in $L_2(\Sbb)$ generates a Hilbert space structure in the space $\Hc$ of these solutions. It is also known that this norm is equivalent to another one, expressed in terms of the \emph{far field pattern}, the asymptotic behavior of functions in $\Hc$ along rays going to infinity. Unfortunately both definitions of the norm in $\Hc$ are nonlocal. In Section \ref{The Space Hc} we present an exposition of these properties.

Section \ref{Sect.Topl} is devoted to the definition of Toeplitz operators.
The nonlocality of the norm in $\Hc$ prevents one from defining a convenient enveloping Hilbert space for $\Hc$. We discuss this topic in Section \ref{Sect.envel}. On the other hand, the definition by means of sesquilinear forms turns out to be quite feasible. This becomes possible due to the reproducing kernel nature of the space $\Hc$. An explicit expression for the corresponding reproducing kernel is given in Section \ref{Sect.sesq}.

Among possible classes of sesquilinear  forms to study, we choose two, which by our opinion deserve certain interest. The first one contains sesquilinear forms $$\Fb(u,v)=\int_{\R^d}a(x)u(x)\overline{v(x)}dx,$$ with $a(x)$ being a bounded function on $\R^d$, which we refer to as `the symbol' of the operator. These forms appear in a number of applications, especially in the classical and quantum scattering theory. Operators generated by such forms are closely related to the so-called Born approximation to the scattering matrix. On the other hand, these forms admit convenient and useful  generalizations to unbounded functions $a(x)$ and even distributions.

Considering Toeplitz operators generated by the forms of the this kind, we start, in Section  \ref{se:radial}, by the analysis of the most transparent, but, nevertheless interesting case when the symbol $a(x)$ is radial. Here, the complete spectral analysis of operators can be performed. Furthermore we give the conditions for boundedness, compactness and membership to Schatten classes.

Passing to general  symbols $a(x)$, in Section \ref{Sect.Special form}, we discuss the representation of the operators in the space $L_2(\Sbb)$ via the Fourier transform of the symbol. Such representation enables us to extend the definition of Toeplitz operators to rather singular symbols, belonging to certain classes of distribution (see Section \ref{Subsect.Reduction}). Then we pass to, probably, the first property which for Toeplitz operators in Helmholtz spaces turn out to be quite different from those considered in the Fock and Bergman spaces, see Section \ref{Sub.Degeneration}. This property concerns uniqueness. Namely, the question is: \emph{to what extent does the operator define its symbol}. In other words, are there any nontrivial symbols which generate the zero operator.  As it turns out, the Toeplitz operators on the Herglotz space can be zero even for  nice fast decaying symbols (but not for compactly supported ones). In Sect \ref{Sub.Degeneration} we explain the mechanism of this phenomenon.

The next topic of our study concerns the finite rank property - see Section \ref{SectFR}. This property, for operators in the Bergman and Fock spaces, has been under an active study recently, and some decisive result have been obtained. Generally, these results assert that if a Toeplitz operator has finite rank then the symbol (moderately growing, to avoid the non-uniqueness case) must be zero or, in a certain sense, highly degenerate. Here, we establish that for symbols $a(x)$ with compact support (again, this condition removes non-uniqueness), the finite rank Toeplitz operator in the Helmholtz space must be zero.

Further on, in  Sections \ref{Sect.Boundedness} and 8, we discuss  conditions for boundedness of Toeplitz operators with general symbols $a(x)$, including the possibility of $a$ being a distribution. The boundedness conditions we found are expressed in terms of the Fourier transform of $a$. More sharp boundedness conditions are obtained for the case of the symbol being majorated by a homogeneous function. A natural consequence of boundedness results are the ones for the compactness of the operator.

In the final Section \ref{h-convolutions} we consider Toeplitz operators defined by a different kind of sesquilinear forms  $$\Fb(u,v)=\int_{\Sbb}a(\x)\f(\x)\overline{\psi(\x)}dS(\x),$$ where $\f,\psi$ are functions in $L_2(\Sbb)$ which represent functions $u,v\in \Hc$. This class of forms is interesting since the corresponding Toeplitz operators form a commutative algebra. An alternative description of such Toeplitz operators is achieved on the base of a new operation, which we call $h$-convolution. Their boundedness conditions, spectral properties, and the structure of invariant subspaces are described in full detail.

For  readers' convenience, we present the situation with non-degeneracy for Toeplitz operators in the Bergman and the Fock spaces  in Appendix \ref{SectDegenFock}.

The first-named author thanks CINVESTAV, Mexico for hospitality and financial support during an essential part of the work with the paper.

\section{The space of Herglotz solutions of the Helmholtz equation}\label{The Space Hc}

We consider the Helmholtz equation
\begin{equation}\label{Helm1}
    \D u+u=0
\end{equation}
in $\R^d$, $d>1$. Among various classes of its solutions, the main object of interest in the paper is  the space of the so-called Herglotz solutions, which are defined as those whose Fourier transform has support in the unit sphere $\Sbb=S^{d-1}\subset\R^d$, moreover, $\Fc u$ is an $L_2$ function on $\Sbb.$ Such solutions possess the slowest possible decay rate at infinity: if a solution decays faster than the Herglotz one, it must be zero. Such solutions can  also be characterized by their asymptotic behavior at infinity: they satisfy the so-called radiation condition. This class of solutions plays an important role in the scattering theory, both classical and quantum. A detailed analysis of these solutions has been performed in \cite{Agmon1}; we review here the results of \cite{Agmon1} which we will use in the paper.

Having a distributional solution  $u\in\Sc(\R^d)$ of \eqref{Helm1}  in the Schwartz space, we can apply the Fourier transform, arriving at $(|\x|^2-1)\hat{u}=0$, therefore $\hat{u}=0$ must be a distribution supported on the unit sphere $\Sbb=S^{d-1}\subset \R^{d}$.
Following \cite{Agmon1}, we consider solutions of \eqref{Helm1} of the form
\begin{equation}\label{GeneralForm}
u(x) = (\Ib \f)(x) = c_d\int_{\Sbb}\f(\x)e^{ix\x}\,dS(\x),  \quad \text{with} \ \
c_d  =  \frac{\sqrt{\pi}}{(2\pi)^{d/2}},
\end{equation}
where $dS$ is the  Lebesgue measure on $\Sbb$ and $\f\in L_2(\Sbb)$.  We denote the scalar product in $L_2(\Sbb)$ by $\lfloor\cdot,\cdot\rfloor$.

The solutions of the form \eqref{GeneralForm}, with $\f\in L_2(\Sbb)$, possess the following basic properties.
First of all, they satisfy the Sommerfeld radiation conditions,
\begin{equation*}
    \int_{|x|=R}|\partial_r u-iu|^2 dS=o(R^{d-1}).
\end{equation*}
Further on, as it is established in \cite[Theorem 4.5 ]{Agmon1}, the asymptotic behavior of $u$  is closely related with the initial function $\f$
in the following mean value sense.

Introduce the space
\begin{equation}\label{AgmonSpace0}
    \Bb^*=\{u\in L_{2,\loc}:\|u\|_*^2\equiv\sup_{R>0}R^{-1}\int_{|x|<R}|u|^2dx<\infty\}.
\end{equation}
In this space, the set $\dot{\Bb}^*$, consisting of functions $u$ such that
\begin{equation}\label{AgmonZero}
    \lim_{R\to\infty}R^{-1}\int_{|x|<R}|u|^2dx=0,
\end{equation}
forms a closed subspace. Then the far field pattern is present,
\begin{equation}\label{Limit}
    u(x)= (2\pi/|x|)^{(d-1)/2}\left(\ve^+e^{i|x|}\f(x/|x|)+\ve^-e^{-i|x|}\f(-x/|x|)\right)+v(x),
\end{equation}
with $v\in \dot{\Bb}^*$, where $\ve^{\pm}=\exp(\pm i(d-1)\pi/4)$.  Since  \eqref{Limit} involves the values of $\f$ in two diametrically opposite points in $\Sbb$, this relation does not immediately provide an expression for $\f(\x)$ in the terms of $u$. Nevertheless, the following  norm equality can be derived  from \eqref{Limit} (see \cite[(4.32)]{Agmon1}, with the coefficient changed according to \eqref{GeneralForm}):
\begin{equation}\label{unitary}
    \|u\|_{\Bb^*}^2=\lim_{R\to\infty} R^{-1} \int_{|x|<R}|u(x)|^2dx=\| \f\|_{L_2(\Sbb)}^2.
\end{equation}
Relation \eqref{unitary} means that with our choice of coefficients, the mapping
\begin{equation}\label{I unitary}
    \Ib:L_2(\Sbb)\to \Hc, \ \Ib: \f\mapsto u,
\end{equation}
is an isometry of the space $L_2(\Sbb)$ to the space $\Hc$ of solutions of \eqref{Helm1}, belonging to $\Bb^*$, with norm \eqref{AgmonSpace0}.  Moreover, it is proved in \cite{Agmon1} that $\Ib$ is an isometry \emph{onto} $\Hc$.

We calculate now the adjoint operator $\Ib^*:\Hc\to L_2(\Sbb)$,  $\Ib^*=\Ib^{-1}$:
\begin{eqnarray*}
 \langle \Ib \f, v \rangle &=& c_d \lim_{R \to \infty} R^{-1} \int_{|x|<R} \int_{\Sbb}\f(\x)e^{ix\x}\,dS(\x) \overline{v(x)} dx \\
&=&c_d \int_{\Sbb} \f(\x) \overline{ \lim_{R \to \infty} R^{-1} \int_{|x|<R}v(x)e^{-ix\x}\,dx}dS(\x).
\end{eqnarray*}
This means that
\begin{equation} \label{eq:I*}
 (\Ib^* v)(\xi) = c_d \lim_{R \to \infty} R^{-1} \int_{|x|<R} v(x)e^{-ix\x}\,dx.
\end{equation}

To describe a convenient orthonormal basis in $\Hc$, we introduce first the standard orthonormal basis in $L_2(\Sbb)$ formed by real-valued spherical harmonics (see, e.g., \cite{Atkin-Han}):
\begin{equation*}
 \left\{\, Y_{n,j}\, : \ n \in \mathbb{Z}_+, \ j = 1,..., N_{n,d}\, \right\},
\end{equation*}
where
\begin{equation*}
 N_{n,d} = \begin{cases}
            1, & \text{if} \ n=0, \\
            \frac{(2n+d-2)(n+d-3)!}{n!(d-2)!}, & \text{if} \ n \geq 1.
           \end{cases}
\end{equation*}
Since
\begin{equation*}
 \delta_{n,m}\delta_{j,l} = \lfloor Y_{n,j}, Y_{m,l} \rfloor = \langle \Ib Y_{n,j}, \Ib Y_{m,l} \rangle,
\end{equation*}
the system of functions
\begin{equation*}
 \mathbf{e}_{n,j} = \Ib Y_{n,j}, \quad \text{with} \ \  n \in \mathbb{Z}_+, \ j = 1,..., N_{n,d},
\end{equation*}
forms an orthonormal basis in $\Hc$. By \cite[p. 37]{Helgason81},  we have
\begin{equation} \label{eq:e_n,j}
 \mathbf{e}_{n,j}(x) = \sqrt{\pi}\, i^n\, \frac{J_{n+(d-2)/2}(r)}{r^{(d-2)/2}}\, Y_{n,j}(\xi),
\end{equation}
where $x =r\xi$, $\xi\in\Sbb$ and $J_{\nu}$ is the Bessel function of order $\nu$.

\section{Definition of Toeplitz operators}\label{Sect.Topl}
\subsection{Enveloping spaces}\label{Sect.envel}
Having in mind to develop a theory of Toeplitz operators in $\Hc$, we will discuss here the possibility of defining such operators in the traditional way. This means that, having an enveloping Hilbert space $\Gc\supset \Hc$, with $\Hc$ being a closed subspace, and the orthogonal projection $\Pb:\Gc\to\Hc$, the Toeplitz operator with symbol $a(x), \, x\in\R^d$, might be defined as
$\Tb_a=\Pb a$, i.e., $\Tb_a : \, u\in\Hc\mapsto \Pb( a u)\in\Hc$.  Some candidates for such an enveloping space were proposed in \cite{Agmon1} and other papers.  For example, the already mentioned space
\begin{equation}\label{AgmonSpace}
    \Bb^*=\{u\in L_{2,\loc}:\|u\|_*^2\equiv\sup_{R>0}R^{-1}\int_{|x|<R}|u|^2dx<\infty\},
\end{equation}
proposed in \cite{AgmonHorm}, might serve as an enveloping space. However it is not a Hilbert space, although the norm defined in \eqref{AgmonSpace}, being restricted to $\Hc$, is equivalent to the Hilbert norm \eqref{unitary} in $\Hc$. To have a Hilbert norm, one might modify \eqref{AgmonSpace}, setting instead
\begin{equation}\label{AgmonNonlocal}
  \Gc=\{u\in L_{2,\loc}:\|u\|_*^2\equiv\limsup_{R>0}R^{-1}\int_{|x|<R}|u|^2dx<\infty\},
\end{equation}
however the expression in \eqref{AgmonNonlocal} degenerates on the subspace $\dot{\Bb}^*$.
Considering the quotient space would lead to a highly inconvenient non-locally defined space, not admitting locally defined operators. Another approach, developed in \cite{Guo}, leads to a certain Banach space with mixed norm, as the enveloping space for $\Hc$.

An enveloping Hilbert space for $\Hc$, in dimension $d=2$, was proposed in \cite{AlvarezFS}, see  also \cite{BarceloFPR}. As shown there, the space $\Wc$ with norm
\begin{equation}\label{HC2}
    \|u\|^2_{\Wc}=\int_{\R^2}\left(|u|^2+|\frac{\partial u}{\partial \vartheta}|^2\right)(1+r)^{-3}dx
\end{equation}
(in polar co-ordinates $x=(r,\vartheta)$ in $\R^2$) contains all solutions of the Helmholtz equation of the form $\Ib \f,$ $\f\in L_2(S^1)$, with the norm equivalent to $\|\f\|_{L_2(S^1)}$. This construction enabled the authors of \cite{BarceloFPR} to investigate properties of some Toeplitz operators in $\Hc$.

In higher dimension, a generalization of the above result was obtained  in \cite{Perez.Herglotz}. The enveloping Hilbert space
has the norm similar to \eqref{HC2}, with replacement of $|\frac{\partial u}{\partial \vartheta}|^2$ by $|{\nabla}_S u|^2$, where $\nabla_S$ denotes the spherical gradient.

Such spaces $\Wc$ turned out to be useful in the study of properties of the Helmholtz spaces. However, these spaces do not admit the multiplication by a non-smooth function, and therefore they are not sufficient for the analysis of  Toeplitz operators with nonsmooth symbols, in the classical setting. We mention as well that the representation of the reproducing kernel, found in \cite{AlvarezFS}, \cite{Perez.Herglotz}, is rather non-explicit.

\subsection{The reproducing kernel}\label{Sect.sesq}
The reproducing kernel for the space $\Hc$ will play an important role in our study. Some formulas for this kernel have been found in \cite{BarceloFPR}, however the representation to follow is more convenient for further analysis. Since the scalar product in $\Hc$ is defined in a rather non-local way, it is more convenient to express the Toeplitz operators in terms of the space  $L_2(\Sbb)$ via the unitary operator $\Ib$.

Let $u(x) = \Ib\f$ be a function in $\Hc$, for some $\f\in L_2(\Sbb)$. Thus, for the reproducing kernel $\k(x,y)=\overline{\k(y,x)}=\k_x(y)$, we have
\begin{equation}\label{repr1}
    u(x)=\langle u(\cdot), \k_x(\cdot)\rangle.
\end{equation}

Since the function $\k_x(\cdot)$ belongs to $\Hc$ for any $x\in\R^d$, it can be represented as
\begin{equation}\label{repr2}
    \k_x(y)=(\Ib \Psi_x)(y) = c_d \int_{\Sbb}\Psi_x(\x)e^{i\x y}dS(\x),
\end{equation}
with some function $\Psi_x\in L_2(\Sbb)$.
Therefore,
\begin{equation}\label{repr3}
    \langle u,\k_x\rangle=\lfloor\f, \Psi_x\rfloor =  \int_{\Sbb}\f(\x)\overline{\Psi_x(\x)}dS(\x)=u(x).
\end{equation}
On the other hand, we have
\begin{equation}\label{repr4}
    u(x) =  c_d \int_{\Sbb}\f(\x)e^{ix\x}dS(\x).
\end{equation}
Due to the arbitrariness of $\f\in L_2(\Sbb)$, \eqref{repr3}, \eqref{repr4} imply that
\begin{equation}\label{repr5}
    \Psi_x(\x) = c_d\, e^{-ix\x},
\end{equation}
and, finally,
\begin{equation}\label{repr6}
    \k_x(y) = c_d \int_{\Sbb}\Psi_x(\xi)e^{iy\x}dS(\x) = c_d^2 \int_{\Sbb}e^{-i(x-y)\x}dS(\x)=\kb(|x-y|).
\end{equation}
The last relation means that the reproducing kernel for $\Hc$ is of difference type, $\k_x(y)=\kb(|x-y|)$,  with $\kb$ being the, obviously, rotation invariant, Fourier transform of the Lebesgue measure on the sphere $\Sbb$ or, in other words, of the distribution $\d(|\x|-1)\in \Ec'(\R^d)$.

An explicit expression for the reproducing kernel follows from \cite[p. 37]{Helgason81}, or \cite[Formula (4.42)]{Agmon1}:
\begin{equation} \label{repr7}
 \kb(|x-y|) = \frac{\pi}{(2\pi)^{d/2}}\, |x-y|^{-(d-2)/2}J_{(d-2)/2}(|x-y|),
\end{equation}
where $J_{(d-2)/2}$ is the Bessel function of order $(d-2)/2$.

We mention as well an alternative formula for the reproducing kernel $\k_x(y)=\kb(|x-y|)$, that uses its series representation in the  the orthonormal basis \eqref{eq:e_n,j}:
\begin{eqnarray*}
 \kb(|x-y|) &=& \sum_{n \in \mathbb{Z}_+, \ j = 1,..., N_{n,d}} \mathbf{e}_{n,j}(y)\overline{\mathbf{e}_{n,j}(x)} \ \ = \sum_{n \in \mathbb{Z}_+, \ j = 1,..., N_{n,d}} \mathbf{e}_{n,j}(x)\overline{\mathbf{e}_{n,j}(y)} \\
 &=& \pi\,\sum_{n \in \mathbb{Z}_+, \ j = 1,..., N_{n,d}}  \frac{J_{n+(d-2)/2}(r)J_{n+(d-2)/2}(\rho)}{r^{(d-2)/2}\rho^{(d-2)/2}}\, Y_{n,j}(\xi)Y_{n,j}(\eta), \quad
\end{eqnarray*}
where $ x =r\xi, \ y =\rho\eta$, $\xi,\eta\in\Sbb.$
\subsection{Sesquilinear forms}
Since there is no commonly accepted convenient enveloping functional Hilbert space  for $\Hc$, we will use another approach to Toeplitz operators, the one based upon considering sesquilinear forms of a special kind.
This approach does  not require any enveloping Hilbert space,  it uses the sesquilinear form and the reproducing kernel only. It has been proposed by the authors in \cite{RozVas1}, having been applied to Toeplitz operators in the Fock space over the complex plane $\C$ and developed further in \cite{RozVas2}  for the Bergman space case on the disk in $\C$. In addition to eliminating the need of an enveloping space, this approach enabled us to consider Toeplitz operators with highly singular symbols, involving measures, distributions and even certain hyper-functions.

As usual, a bounded sesquilinear form  $\Fb(u,v)$ in $\Hc$ is linear in $u$, anti-linear in $v$, and satisfies
$|\Fb(u,v)|\le C\|u\|\|v\|$ for all $u,v\in\Hc.$
As explained in  \cite{RozVas1}, having such a bounded sesquilinear form $\Fb$, the Toeplitz operator $\Tb_\Fb$, with form-symbol $\Fb$, in $\Hc$ is defined by
\begin{equation}\label{operator}
   ( \Tb_\Fb u)(x)=\Fb(u, \k_x(\cdot)),
\end{equation}
where $\k_x$ is the reproducing kernel for the space $\Hc$. Further on, we will consider several types of sesquilinear forms and study the properties of the corresponding operators.

\section{Radial Toeplitz operators} \label{se:radial}

Before presenting  in the subsequent sections the detailed study of Toeplitz operators defined by general sesquilinear forms
$$\Fb_a(u,v)=\int_{\R^d} a(x)u(x)\overline{v(x)}dx, \ u,v\in\Hc,$$
we consider here its special, important and more transparent, case when the function $a$  is \emph{radial}: $ a = a(|x|) =a(r)$.

Under suitable conditions on a locally integrable radial function $a = a(|x|) = a(r)$, the \emph{radial} Toeplitz operator $\Tb_a := \Tb_{\Fb_a}$, defined by the sesquilinear form
\begin{equation*}
 \Fb(u,v)=\Fb_a(u,v)=  \int_{\R^d} a(|x|)u(x)\overline{v(x)}dx, \ u,v\in\Hc,
\end{equation*}
is given by \eqref{operator}:
\begin{equation*}
 (\Tb_au)(x) = \int_{\R^d} a(|y|)u(y)\overline{\k_x(y)}\,dy.
\end{equation*}
For each $\mathbf{e}_{n,j}(x)$, we calculate formally ($y =\rho\eta$)
\begin{eqnarray*}
 (\Tb_a\mathbf{e}_{n,j})(x) &=& \int_{\R^d} a(|y|)\mathbf{e}_{n,j}(y)\overline{\k_x(y)}\,dy \\
 &=& \int_{\R^d} a(|y|)\mathbf{e}_{n,j}(y)\left(\sum_{m \in \mathbb{Z}_+, \ l = 1,..., N_{n,d}} \mathbf{e}_{m,l}(x)\overline{\mathbf{e}_{m,l}(y)}\right)\,dy \\
 &=& \sum_{m \in \mathbb{Z}_+, \ l = 1,..., N_{n,d}} \mathbf{e}_{m,l}(x) \int_{\R^d} a(|y|)\mathbf{e}_{n,j}(y) \overline{\mathbf{e}_{m,l}(y)}\,dy \\
 &=& \sum_{m \in \mathbb{Z}_+, \ l = 1,..., N_{n,d}} \mathbf{e}_{m,l}(x)\,  \pi \int_{\mathbb{R}_+} a(\rho) \frac{J_{n+(d-2)/2}(\rho)J_{m+(d-2)/2}(\rho)}{\rho^{d-2}}\,\rho^{d-1}\, d\rho \\
 && \times \ \int_{\Sbb}Y_{n,j}(\eta) \overline{Y_{m,l}(\eta)}\, dS(\eta) \\
 &=& \left(\pi \int_{\mathbb{R}_+} a(\rho) [J_{n+(d-2)/2}(\rho)]^2 \rho\, d\rho \right)\,
 \mathbf{e}_{n,j}(x).
\end{eqnarray*}

Given a locally integrable radial function $a = a(|x|) = a(r)$, we introduce the \emph{spectral sequence}
$\pmb{\gamma}_a =\{\gamma_a(n)\}_{n \in \mathbb{Z}_+}$, where
\begin{equation}\label{radial.spectral}
 \gamma_a(n) = \pi \int_{\mathbb{R}_+} a(r) [J_{n+(d-2)/2}(r)]^2\, r\, dr.
\end{equation}
Then the above calculation leads to the following assertion.

\begin{theorem}
Let  a locally integrable radial function $a = a(|x|) = a(r)$ be such that the corresponding sequence $\pmb{\gamma}_a =\{\gamma_a(n)\}_{n \in \mathbb{Z}_+}$ is bounded.
Then the Toeplitz operator $\Tb_a$, initially  densely defined {\rm (}on  finite linear combinations of basis elements $\mathbf{e}_{n,j}${\rm )},  admits the bounded extension by continuity  to the whole  Hilbert space $\Hc$, and possesses the following properties:
\begin{itemize}
 \item [-] the Toeplitz operator $\Tb_a$ is diagonal with respect to the orthonormal basis $\{\mathbf{e}_{n,j}\}$;
 \item [-] each $\gamma_a(n)$, $n \in \mathbb{Z}_+$, is an eigenvalue of $T_a$ with  multiplicity $N_{n,d}$, whose corresponding eigenspace has the form
 \begin{equation*}
  \mathbb{H}_n^d = \mathrm{span}\,\{\mathbf{e}_{n,1}, \mathbf{e}_{n,2}, ..., \mathbf{e}_{n,N_{n,d}}\};
 \end{equation*}
 \item [-] the spectrum of the Toeplitz operator $\Tb_a$ is given by
\begin{equation*}
 \mathrm{spec}\,\Tb_a = \mathrm{closure}\, \{\gamma_a(n)\}_{n \in \mathbb{Z}_+};
\end{equation*}
 \item [-] the Toeplitz operator $\Tb_a$ is compact if and only if $\pmb{\gamma}_a \in c_0$.
\end{itemize}
\end{theorem}

\begin{corollary}
 The $C^*$-algebra generated by such Toeplitz operators is commutative.
\end{corollary}

A simple sufficient condition for boundedness of the operator $\Tb_a$ is given by  the following lemma.
\begin{lemma} \label{le:L_1}
 For any $a = a(r) \in L_1(\mathbb{R}_+)$, the Toeplitz operator $\Tb_a$ is bounded.
\end{lemma}

\begin{proof}
 Follows directly from the well-known asymptotic formula for the Bessel function:
\begin{equation*}
 J_{\nu}(r) = \sqrt{\frac{2}{\pi r}}\, \cos\left(r -\textstyle{\frac{\pi \nu}{2}- \frac{\pi}{4}}\right) + O(r^{-3/2}), \quad \mathrm{as} \ \ r \to \infty.   \qedhere
\end{equation*}
\end{proof}

The numbers $\gamma_a(n)$, being eigenvalues of $\Tb_a$, do not form the eigenvalue sequence in the sense usually adopted in the spectral theory of operators: namely, they are not numerated in the nonincreasing (for the positive ones) or nondecreasing (for the negative ones) order. The spectral sequence $\pmb{\gamma}_a$ has an intrinsic ordering and therefore contains more information than only on the spectrum of the operator $\Tb_a$. It is an interesting question, which sequences of numbers $\pmb{\g}$ may serve as the spectral sequence of an operator $\Tb_a$, with  certain symbol $a(r)\in L_1(\R_+)$.  For a similar problem for radial Toeplitz operators in the Bergman space on the disk, such complete description has been found in \cite{BauHerVas}.

We give now two illustrative examples, where the sequence $\pmb{\gamma}_a$ can be explicitly calculated. In particular, they will show that the condition of Lemma \ref{le:L_1} is not necessary for the boundedness of $T_a$.

\begin{example}
 Let $a(r) = \sin(\textstyle{\frac{r^2}{4}})$, then, by \cite[Formula 6.729.1]{Grad}, we have
\begin{equation*}
\gamma_a(n) = \pi \int_{\mathbb{R}_+} \sin(\textstyle{\frac{r^2}{4}}) [J_{n+(d-2)/2}(r)]^2\, r\, dr =
2\pi \cos(1 -\textstyle{\frac{\pi(n+(d-2)/2)}{2}})\, J_{n+(d-2)/2}(2).
\end{equation*}
Recall that
\begin{equation*}
 J_{n+(d-2)/2}(2) = \sum_{k=0}^{\infty} (-1)^k \frac{1}{k!\Gamma(n+d/2+k)},
\end{equation*}
thus
\begin{equation*}
\frac{1}{\Gamma(n+d/2)} - \frac{1}{\Gamma(n+d/2+1)} \ < J_{n+(d-2)/2}(2) \ < \frac{1}{\Gamma(n+d/2)}.
\end{equation*}
That is, the operator $\Tb_a$ is compact with the following asymptotic behavior of its spectral sequence
\begin{equation*}
 \gamma_a(n) \ \approx \ 2\pi {\cos(1 -\textstyle{\frac{\pi(n+(d-2)/2)}{2}})}\,
 \frac{1}{\Gamma(n+d/2)}, \quad \mathrm{as} \ \ n \to \infty.
\end{equation*}
\end{example}

\begin{example}
 Let $a(r) = r^{-\mu}$, where $1 < \mu < d$ (the right inequality is needed to guarantee the convergence in a neighborhood of 0 of the integral \eqref{radial.spectral}  defining  $\gamma_a(n)$ ). Then, by \cite[Formula 6.574.2]{Grad}, we have
\begin{equation*}
 \gamma_a(n) = \pi\, \int_{\mathbb{R}_+} r^{-\mu}[J_{n+(d-2)/2}(r)]^2\, r\, dr =
  \frac{\pi\Gamma(\mu-1)\Gamma(n+(d-\mu)/2)}{2^{\mu-1}[\Gamma(\mu/2)]^2\Gamma(n+(d+\mu)/2 -1)}.
\end{equation*}
Asymptotically:
\begin{equation*}
 \gamma_a(n) \ \approx \ \frac{\pi\Gamma(\mu-1)}{2^{\mu-1}[\Gamma(\mu/2)]^2}\, \frac{1}{(n+d/2 -1)^{\mu -1}}, \quad \mathrm{as} \ \ n \to \infty.
\end{equation*}
 This implies that the operator $\Tb_a$ is compact. Moreover, the operator $\Tb_a$ belongs to the Schatten class $\frak{S}_p$, $1 \leq p < \infty$, if and only if $\mu > (p+1)/p$.
\end{example}

Note that the conditions in Lemma \ref{le:L_1} and Example 4.5 are almost sharp for sign-definite symbols $a$: for $a(r)=r^{-1}$ the integral defining $\g_a(n)$ diverges for all $n\in\Z_+$ and the operator is, of course, unbounded. On the other hand, Example 4.4 demonstrates that a rapid oscillation of a \emph{non}--$L_1$-symbol may still produce a bounded and even compact operator due to the cancellations  in the integral \eqref{radial.spectral}. For operators in the Bergman space this cancellation effect has been studied in \cite{Vasil}.

\section{Sesquilinear forms  $\Fb_a(u,v)=\int a(x)u(x)\overline{v(x)}dx$}\label{Sect.Special form}
\subsection{Reduction to the sphere}\label{Subsect.Reduction}
Consider the sesquilinear form
\begin{equation}\label{Form1}
 \Fb(u,v)=\Fb_a(u,v)=  \int_{\R^d} a(x)u(x)\overline{v(x)}dx, \ u,v\in\Hc,
\end{equation}
with a locally integrable function $a(x)$, further restrictions to be fixed later on. The operator in $\Hc$ defined according to \eqref{operator} by this sesquilinear form will be denoted $\Tb_a:=\Tb_{\Fb_a}$, and the function $a$ itself will be called the \emph{symbol} of the sesquilinear form \eqref{Form1} and of the corresponding Toeplitz operator.

The forms \eqref{Form1} play an important role in the scattering theory, both classical (e.g., acoustic) and quantum, see, for example \cite{Colton}, \cite{RodTh}, \cite{Yafa}, where they generate operators of the so-called \emph{Born approximation}.
Since $u,v\in\Hc$, these functions admit the representations $u=\Ib\f$, $v=\Ib\psi$, with $\f,\psi\in L_2(\Sbb)$. We substitute these representations into \eqref{Form1}, to obtain
\begin{equation}\label{Form2}
    \Fb_a(u,v)= c_d^2 \int_{\R^d}\int_{\Sbb}\int_{\Sbb}a(x)e^{i(\x-\y)x}\f(\x)\overline{\psi(\y)}dxdS(\x) dS(\y).
\end{equation}
Considered in the sense of distributions, the change of the order of integration in \eqref{Form2} is legal. We denote by $\widehat{a}$ the Fourier transform of the symbol $a$,
\begin{equation*}
 \widehat{a}(\x) = (2\pi)^{-d/2} \int_{\R^d} e^{-ix\x}a(x)\,dx,
\end{equation*}
then
\begin{equation}\label{Form3}
    \Fb_a(u,v)= \frac{\pi}{(2\pi)^{d/2}}\int_{\Sbb}\int_{\Sbb}\widehat{a}(\y-\x)\f(\x)\psi(\y)dS(\x)dS(\y)=\lfloor
    \breve{\Tb}_a\f,\psi \rfloor,
\end{equation}
where $\breve{\Tb}_a:=\breve{\Tb}_{\Fb_a}$ is the operator in $L_2(\Sbb)$, defined by
\begin{equation}\label{Operator1}
   (\breve{\Tb}_a\f)(\y)= \frac{\pi}{(2\pi)^{d/2}}\int_{\Sbb} \widehat{a}(\y-\x)\f(\x)dS(\x).
\end{equation}

This means that the operator $\Tb_a$, acting on $\Hc$, is unitary equivalent to the integral operator $\breve{\Tb}_a$  with kernel $\hat{a}(\y-\x)$, acting on $L_2(\Sbb)$. Note that the operator $\breve{\Tb}_a$ acts on functions defined on the sphere $\Sbb$, while the values of the kernel $\widehat{a}(\y-\x)$, that are involved in the action of the operator, are calculated at all points $\y-\x\in \overline{\Bb(0,2)}$, where $\overline{\Bb(0,2)}$ is the closed  ball in $\R^d$ with radius $2$ and centered at the origin. So,
the operator \eqref{Operator1} is not an operator with difference kernel (i.e., not a convolution operator), in the standard sense. To illustrate this observation, we present the formula by which the operator \eqref{Operator1} acts, for the most simple case of $d=2$. With the usual parametrization of the unit circle, $\x=e^{i\s}$ (correspondingly, $\y=e^{i\vs}$), the operator $\breve{\Tb}_a$ acts as
\begin{equation}\label{Operator2}
    (\breve{\Tb}_a \f)(e^{i\s})= \frac{1}{2} \int_{0}^{2\pi}\widehat{a}(e^{i\s}-e^{i\vs})\f(e^{i\vs})d\vs, \quad \s,\vs\in[0,2\pi].
\end{equation}

\subsection{Degeneration}\label{Sub.Degeneration} The form of the operator $\breve{\Tb}_a$ enables us to reveal some interesting properties, even before specifying the class of the symbols $a$ granting the boundedness or compactness of the operator. Some of these properties differ essentially from the ones for the classical Toeplitz operators in the Bergman or the Fock spaces.

First of all,  since the values of $\widehat{a}$ that are involved in the expression for the action of ${\breve{\Tb}}_a$ are only the ones that are attained at the points of  $\overline{\Bb(0,2)}$, the operator ${\breve{\Tb}}_a$ is zero and, therefore, $\Tb_a = 0$,  as soon as the support of $\hat{a}$ does not intersect the ball $\overline{\Bb(0,2)}$. This implies that, unlike the cases of Toeplitz operators in the Bergman and the Fock spaces, the set of symbols generating the zero operator is rather large. To make the last statement more concrete,  we describe here a general construction of such symbols.

Let $a(x)$, $x\in\R^d$, be a symbol, for example, in $L_p(\R^d), \, p\in[1,2]$. Consider a smooth cut-off function $\o(\x)$, $\x\in \R^d$ such that $\o(\x)=1, |\x|\le2$, and $\o(\x)=0, |\x|\ge R_0>2$. 
The inverse Fourier transform $\check{\o}(x)$ of $\o$ is a smooth function on $\R^d$ with exponential decay. Therefore, the convolution $a_\o=a*\check{\o}$ makes sense, moreover, $a_\o$ belongs to the same class $L_p(\R^d)$ as $a$.

Now consider the symbol $a^\circledast=a-a_\o$. The Fourier transform $\widehat{a^\circledast}(\x)$ equals
\begin{equation*}
    \widehat{a^\circledast}(\x)=\hat{a}(\x)-\hat{a}(\x)\o(\x)=\hat{a}(\x)(1-\o(\x)),
\end{equation*}
and, therefore, $\widehat{a^\circledast}(\x)=0$ on $\overline{\Bb(0,2)}$. We arrive thus at a simple  statement:
\begin{proposition}\label{PropNDegen} Given any $a\in L_p, \, p\in[1,2]$, the Toeplitz operator in the space $\Hc$ with symbol ${a^\circledast}$ is the zero operator.
\end{proposition}
An important consequence is the following corollary.
\begin{corollary} The Toeplitz operator with symbol $a$ coincides with the Toeplitz operator  with symbol $a_{\o}$. In particular, in the study of Toeplitz operators generated  by sesquilinear forms \eqref{Form1}, it is sufficient to consider infinitely differentiable symbols $a$.
\end{corollary}
Note that the situation with non-degeneracy for Toeplitz operators in the Bergman and the Fock spaces is quite different from the above discussed. For the readers' convenience, we present the corresponding results in Appendix \ref{SectDegenFock}.

\subsection{Non-degeneracy classes of symbols}  The smallest of such classes consists of compactly supported symbols.
\begin{proposition}\label{PropCompSupp}Let $a\in L_{1,\comp}$. If $\Fb_a(u,v)=0$ for all $u,v\in\Hc$ then $a\equiv 0$.
\end{proposition}
\begin{proof} First, we remark that the Fourier transform of $a$ is an entire  analytical function of $\x$ variable. On the other hand, the operator defined by the  sesquilinear form $\Fb_a$   is unitary equivalent to the integral operator \eqref{Operator1} with smooth kernel. The latter can be the zero operator only if the kernel $\widehat{a}(\y-\x)$ vanishes for all $|\x|=|\y|=1$. This means that the Fourier transform $\widehat{a}(\x)$ vanishes in the ball $|\x|\le2$, which is impossible for a nonzero analytical function.
\end{proof}
It is easy to see that the above result as well as the reasoning extends to any class of symbols with quasi-analytical Fourier transform. This, in particular, holds for those symbols $a$ that satisfy the condition  $|a(z)|=O(\g(|z|))$ for $|z|\to\infty$ with a monotone function $\g(t)$ satisfying
\begin{equation*}
    \int_0^\infty \frac{|\log(\g(t))|}{1+t^2}dt<\infty.
\end{equation*}

\section{The finite rank property}\label{SectFR}
We continue with the study of  a sharper version of the (non)-degeneracy property. The finite rank problem for Toeplitz operators has been under an active study since 2008, with final results for operators in the Fock space being obtained just recently, see \cite{RozShir} and references therein. The problem consists in describing all symbols in a prescribed class, that generate finite rank Toeplitz operators. Typical  results here for the Bergman and Fock spaces  are the following. For a \emph{functional} symbol with moderate growth (so that the degeneracy does not occur, see Appendix \ref{SectDegenFock}) a finite rank Toeplitz  operator is, in fact, the zero operator, whose symbol is automatically zero. For a distributional symbol in an explicitly described class (in particular, for distributions with compact support) the symbol of a finite rank Toeplitz operator is a finite linear combination of $\delta$-distributions  and their derivatives, centered at finitely many points.

In order to study the finite rank problem for Toeplitz operators in
the space $\Hc$, we need  similar results for operators in the Fock space
 $\Fs^{\HH}=\Fs^{\HH}(\R^m)$, obtained in \cite{AlexRoz}. The space $\Fs^{\HH}$ consists of harmonic functions in $\R^m$, square integrable against the Gaussian measure.
\begin{theorem}\label{ThFRHarm} Let $b(x)$ be a function in $L_{1,\comp}(\R^m)$, $m\ge2$. Suppose that the Toeplitz operator $\Tb_b^{\HH}$ with symbol $b$ in the harmonic Fock space has finite rank. Then $b=0$.
\end{theorem}
The case of an even dimension $m=2d$ is trivial since here, under the identification of $\R^{2d}$ with $\C^d$, the space of analytical functions is a closed subspace in the space of harmonic functions. \\
The proof of Theorem \ref{ThFRHarm}, given in \cite{AlexRoz}, for the case of an odd dimension $m$  is based upon the following simple but important observation, which has been used systematically in all studies on the finite rank problem, starting from \cite{Lue}. Namely, the Toeplitz operator in a Hilbert space $\HF$, determined by the sesquilinear form $\Fb(u,v)$, has the finite rank $N$ if and only if for any collections of functions $u_j,v_k\in \HF$, the rank of the infinite matrix $(\Fb(u_j,v_k))_{ j,k \in \mathbb{N}}$ is not greater than $N$, moreover, $N$ equals the maximal value of this rank, taken over all such collections.

Having this observation in mind, the proof of Theorem \ref{ThFRHarm} is based upon a special choice of the collections of functions $u_j,v_k$, actually, harmonic polynomials not depending on one variable, which enables the dimension reduction. Moreover, the essential (but formal) extension of Theorem \ref{ThFRHarm} replaces the finite rank condition for the operator by the finite rank condition for the matrix $(\Fb(u_j,v_k))_ {j,k\in \mathbb{N}}$ with a particular selection of the functions $u_j,v_k.$
\begin{corollary}\label{CorrHarm}Let $a\in L_{1,\comp}(\R^m)$, and the matrix $(\Fb_a(u_j,v_k)),\, j,k\in \mathbb{N}$ with $\{u_j\},\{v_k\}$ being any collections of  harmonic polynomials in $\R^m$, has finite rank. Then $a=0$.
\end{corollary}
The statement follows immediately from Theorem \ref{ThFRHarm}, since  harmonic polynomials are dense in the space $\Fs^{\HH}$.  This approximation, due to the elliptic regularity, is uniform on compacts. Therefore the value of the  form $\Fb_a(u,v)$ for any harmonic functions in $\Fs^{\HH}$ can be approximated by its values on harmonic polynomials. Consequently, the largest possible rank of the matrix $(\Fb_a(u_j,v_k))$ is attained on collections of harmonic polynomials.

This line of reasoning leads to the basic finite rank result for Toeplitz operators in our space $\Hc$.
\begin{theorem}\label{ThFRHelmh}Let the function $a(x)$ belong to
$L_{1,\comp}(\R^d)$, $d>2$. Suppose that the Toeplitz operator in $\Hc$, defined by the sesquilinear form \eqref{Form1}, has a finite rank. Then $a(x)=0$ a.e. in $\R^d$.
\end{theorem}
\begin{remark}  In \cite{RozTObzor}, another version of Theorem \ref{ThFRHelmh} was proved for the case of the space of solutions of the Helmholtz equation in a bounded domain. The proof of Theorem \ref{ThFRHelmh}  requires certain additional consideration. We present here  its detailed proof enabling further extensions of the result. Note, however, that  we still impose the restriction $d>2$. It will be seen from the proof, at which point this restriction actually becomes critical. Unfortunately, the authors do not possess a proof of the property under discussion for dimension 2. Note also that for the case $N=0$, Theorem \ref{ThFRHelmh} generalizes the reasoning  in the beginning of  Appendix \ref{SectDegenFock}.
\end{remark}

We need to consider some additional spaces of solutions of the Helmholtz equations, larger ones than $\Hc$, actually, with milder restrictions for the behavior at infinity than in $\Hc$. These spaces were introduced  by S. Agmon in \cite{Agmon2}; that paper contains also the main theorems that we are going to use.

Denote by $\Hc_s(\R^d)$ the space of solutions of the Helmholtz equation $\D u+u=0$, satisfying

\begin{equation}\label{Hs estimate1}
    \|u\|_{\Hc_s}^2:=\int_{\R^d}(1+|x|)^{2s}|u(x)|^2dx
    <\infty.
\end{equation}
This is a closed subspace in the weighted $L_2$--space of all functions $u$ satisfying \eqref{Hs estimate1}.

Further on, by  Theorem 6.1 in \cite{Agmon2}, any solution of the Helmholtz equation, satisfying the estimate
\begin{equation}\label{Hs estimate2}
    |u(x)|=O((1+|x|)^N)
\end{equation}
for some $N>0$, admits a representation
\begin{equation}\label{AgmonRepr}
  u(x) =(\varPhi(\cdot),e^{ix\cdot}),
\end{equation}
with some distribution $\varPhi\in \Ec'(\Sbb)$.

In formula \eqref{AgmonRepr}, the right-hand side should be understood in the sense of the action of the distribution $\varPhi$ on the sphere $\Sbb$ on the function $e^{ix\x}$, considered as a smooth function on $\Sbb$ depending smoothly on the parameter $x$. To understand the role of the exponent $N$ in \eqref{Hs estimate2}, one should recall that, due to the compactness of $\Sbb$, any distribution on $\Sbb$ has finite order, which means that it involves differentiation of $e^{ix\x}$ of some bounded order only, which leads to a bounded power of $x$ in the result.

A more exact relation between the singularity of the distribution $\varPhi(\x)$ and the corresponding Helmholtz function $u(x)$ given by \eqref{AgmonRepr}, is described  by Theorem 6.2 in \cite{Agmon2}. We cite here the particular case we are interested in.
\begin{theorem}\label{Agmon6.2} Let $u(x)$ be a solution of the Helmholtz equation, admitting the representation \eqref{AgmonRepr} with some distribution $\varPhi\in\Ec'(\Sbb)$. Then{\rm:}
\begin{itemize}
\item if $\varPhi\in H^{-s}(\Sbb)$ with some $s>0$, then $u\in \Hc_{-s-\frac12}(\R^d)$;
    \item conversely, if the solution $u(x)$ belongs to $\Hc_{-s-\frac12}(\R^d)$ for some $s>0$, then the distribution $\varPhi$ belongs to $H^{-s}(\Sbb)$.
\end{itemize}
Here $H^{-s}(\Sbb)$ is the negative order Sobolev space on the sphere $\Sbb$, with the usual norm. The norms of $\varPhi$ and $u$ in the corresponding spaces are equivalent.
\end{theorem}

Using the above results, we are able to prove the following approximation lemma.
\begin{lemma}\label{plane wave} Let $U(x)=U(x_1,x')$ be a function of the form $U(x)=e^{ix_1}P(x')$,
where $P$ is a harmonic polynomial of the variable $x'\in\R^{d-1}$: $\D_{d-1}P=0$. Then for any $R>0$, the function $U$ can be uniformly approximated by functions $u\in\Hc$ on the ball  $\overline{\Bb(0,R)}$ in $\R^d$ with radius $R$.
\end{lemma}
\begin{proof}
The conditions of the lemma imply that the function $U(x)$ satisfies the Helmholtz equation. If $N$ is the degree of the polynomial $P(x')$, then the function $U(x)$ satisfies the estimate \eqref{Hs estimate2}, with the same $N$, and therefore admits a representation \eqref{AgmonRepr} with some distribution $\varPhi\in\Ec'(\Sbb)$. This distribution can be even explicitly found by means of the Fourier transform:
\begin{equation}\label{varPhi2}
    \varPhi=(2\pi)^{-\frac d2}\delta(\x_1-1)\otimes P(D_{\x'})\d(\x')
\end{equation}
in co-ordinates $\x=(\x_1,\x')$. Thus, $\varPhi$ is a distribution with one-point support at $(1,0)$. The distribution $\varPhi$, being the result of application of an order $N$ differential operator to the $\d$-distribution, belongs to the negative order Sobolev space $H^{-s}(\Sbb)$, with any $s> \frac{d}2+N$. Like any distribution in $H^{-s}(\Sbb)$, our $\varPhi$ can be approximated in the  $H^{-s}(\Sbb)$-norm by smooth functions $\varPhi_\e$, $\e\to0$, for example, using the localized  convolution with a proper family of smooth functions. Since $\varPhi_\e\to\varPhi$ in $H^{-s}(\Sbb)$, by Theorem \ref{Agmon6.2}, the corresponding solutions $U_{\e}$ of the Helmholtz equations, defined by $U_{\e}(x)=(\varPhi_\e,e^{-ix\cdot} )$, converge to $U(x)$ in the norm of the weighted space $\Hc_{-s-\frac12}(\R^d)$. Therefore, for any $R>0$, the functions $U_{\e}$ converge to $U$ in $L_2(\overline{\Bb(0,R)})$, which, by elliptic regularity, implies the uniform convergence.
\end{proof}

Now we proceed  with the proof of Theorem \ref{ThFRHelmh}.
  \begin{proof}Fix some $\x^0\in\,\Sbb$. Pass then to the co-ordinate system in $\R^d$, for which $\x_0=(1,0,\dots,0)$. Consider the systems of functions $u_j=e^{ix_1}P_j(x'), v_k=e^{ix_1}P_k(x')$, where the polynomials $P_j$ form the standard basis in the space of harmonic polynomials in $d-1$ variables.

With the functions $u_j,v_k$, thus chosen, the sesquilinear form $\Fb_a$ becomes
  \begin{equation}\label{Transformed form}
  \Fb_a(u_j,v_k)=\int_{R^{d-1}}\left(\int_{-\infty}^{\infty}a(x_1,x')dx_1\right)P_j(x')\overline{P_k(x')}dx'.
  \end{equation}
The expression \eqref{Transformed form} is nothing but the value of the sesquilinear form ${\Fb}_{\tilde{a}}(P_j,P_k)$ in the $d-1$-dimensional space on the functions $P_j,P_k$,  where $\tilde{a}(x')=\int_{-\infty}^{\infty}a(x_1,x')dx_1$.
This system of functions does not  fit the conditions of Theorem \ref{ThFRHelmh}  immediately, since the functions $u_j=e^{ix_1}P_j(x')$ do not belong to the space $\Hc$. However, such functions belong to some weighted spaces $\Hc_{-s}$ for some $s$ depending on the degree of polynomials,and therefore,  by Lemma \ref{plane wave}, the function $u_j(x)$ can be uniformly, on the compact set (the support of $a$), approximated by functions $u_{j,\e}\in \Hc$. Thus
\begin{equation}\label{limit form}
 \lim_{\e\to0}\int_{\R^d}a(x)u_{j,\e}(x)\overline{u_{k,\e}}(x)dx=\int_{\R^{d-1}}\tilde{a}(x')P_j(x')\overline{P_k(x')}dx'.
\end{equation}

Since the rank of the infinite matrix of the values of  our sesquilinear form on $(u_{j,\e},u_{k,\e})$ is not greater than $\rank (\Tb_a)$, the same holds in the limit as $\e\to 0$, therefore the rank of the infinite matrix $\Fb_a(u_j,u_k)$ is finite. However, by \eqref{Transformed form}, this means that the rank of the infinite matrix of the values of the sesquilinear form ${\Fb}_{\tilde{a}}$ in the $(d-1)$-dimensional space, considered on the harmonic polynomials, is finite. By Corollary \ref{CorrHarm}, the function $\tilde{a}(x')$ must be zero for all $x'$.

Now, recall that the direction of the axis $\x_1$ (and, therefore, the axis $x_1$) was chosen arbitrarily. Therefore, the reasoning, just presented, gives us that the integral of $a(x)$ along any straight line in $\R^d$ is zero. In other words, the function $a$ has zero Radon transform. By the uniqueness theorem for the Radon transform, we obtain $a(x)\equiv0$.

To conclude the proof, we pinpoint the place where the reasoning breaks down for the two-dimensional case, $d=2$. Recall that, having directed the $\x_1$ axis in the chosen direction, we  considered the system of test functions $u_j=e^{ix_1}P_j(x')$, where $P_j(x')$ are harmonic polynomials. For $d>2$ there are infinitely many  linearly independent harmonic polynomials in $\R^{d-1}$, and this circumstance enables the approximation, upon which the proof is based. On the other hand, in dimension $d-1=1$, the space of harmonic polynomials is only two-dimensional, therefore any operator, in particular, a Toeplitz one, in this space has finite rank, not greater than 2, and no consequences for the symbol can be derived from this property.
 \end{proof}

 \section{Boundedness conditions}\label{Sect.Boundedness}
 \subsection{Boundedness and restrictions} In this section we start studying the sufficient conditions on the symbol $a(x)$ for the boundedness of the sesquilinear form \eqref{Form1}. In this analysis, we are going to use both representations \eqref{Form1} and \eqref{Form3}.

 The first statement is, actually, a version of the Young inequality. Recall that this inequality declares that the convolution operator with a function in $L_1$ is a bounded operator in $L_2$. Our operator \eqref{Operator1} is not a convolution operator, therefore the boundedness theorem has a somewhat different, more complicated, formulation.
 \begin{theorem}\label{Criter.bounded}For a point $\x\in \Sbb, \ |\x|=1,$ denote by $S_{\x}$ the unit sphere $\{\y: \, |\y-\x|=1 \},$ centered at $\x$. Suppose that for any $\x\in\Sbb$, the restriction of  $\widehat{a}$ to $S_{\x}$ belongs to $L_1(S_{\x})$ and, moreover, $\|\widehat{a}\|_{L_1(S_{\x})}\le C$. Then the operator \eqref{Operator1} is bounded in $L_2(\Sbb)$, and, therefore, the operator $\Tb_a$, is bounded in $\Hc$. Moreover $\|\Tb_a\|=\|\breve{\Tb}_a\|\le C$.
 \end{theorem}
 \begin{proof}We follow mainly the standard proof of the classical Young inequality. To establish the result, we use the Riesz-Thorin theorem. First, if $\f\in L_{\infty}(\Sbb)$, then
 \begin{equation}\label{HY1}
  |(\breve{\Tb}_a\f)(\x)|=\left|\frac{\pi}{(2\pi)^{d/2}}\int_{\Sbb}\widehat{a}(\y-\x)\f(\y)dS(\y)\right|\le \frac{\pi}{(2\pi)^{d/2}}\,\|\f\|_{L_{\infty}(\Sbb)}\int_{S_\x}|\widehat{a}(\z)|dS(\z),
 \end{equation}
 therefore $\breve{\Tb}_a$ is a bounded operator in $L_{\infty}(\Sbb)$ with norm not greater than $C$. Next, the same reasoning applied to the adjoint operator $\breve{\Tb}_a^*$ with kernel $\overline{\widehat{a}(\x-\y)}$, justifies the boundedness of $\breve{\Tb}_a^*$ in $L_{\infty}(\Sbb)$. This implies the boundedness of $\breve{\Tb}_a$ in $L_1(\Sbb)$, with the same norm estimate. Finally, the Riesz-Thorin theorem grants the boundedness of our operator in $L_2(\Sbb)$.
 \end{proof}
 Further on,  we will see some more  relations between the properties of our Toeplitz operators and restriction theorems of Stein-Thomas type.

\subsection{Distributional symbols} Here we collect some more general statements concerning the boundedness of Toeplitz operators generated by the sesquilinear forms \eqref{Form1}. It is convenient to generalize here the expression \eqref{Form1} by introducing symbols-distributions.

We start with some remarks. For a Schwartz distribution $a\in\Sc'(\R^d)$, the action $(a,\Phi)$ of the distribution $a$ on the function $\Phi$ is initially defined only for $\Phi$ in the Schwartz space $\Sc(\R^d)$. However, if  $a$ is a continuous functional with respect to some particular Schwartz norm,
 \begin{equation}\label{Schwartz1}
    |(a,\Phi)|\le C \sup_{x\in\R^d;|\a|,|\b|\le N} \langle x\rangle^\b |D^\a\Phi(x)|,\quad \langle x\rangle=1+|x|,
 \end{equation}
then the distribution $a$ can be extended by continuity to the Banach space of all functions whose norm \eqref{Schwartz1} is finite.

 \begin{definition}Let $a$ be a Schwartz  distribution in $\R^d$, $a\in \Sc'(\R^d)$. We define
 \begin{equation}\label{Form1.dist}
    \Fb_a(u,v)=(a,u\overline{v}),
 \end{equation}
for such pairs $u,v\in\Hc$, for which the action of the distribution $a$ on the smooth function $u\overline{v}$ is defined.
 \end{definition}
 If for a particular distribution $a$, the sesquilinear form \eqref{Form1.dist} turns out to be defined for all functions $u,v\in \Hc$ and bounded in $\Hc$, it determines an operator in $\Hc$, which we refer to as the Toeplitz operator associated with $a$,
 \begin{equation}\label{Distr}
    \langle\Tb_a u, v\rangle=\Fb_a(u,v)=(a,u\overline{v}).
 \end{equation}
 As usual, by means of the reproducing kernel $\k_x(y)=\k(x,y)=\kb(|x-y|)$, found in \eqref{repr7}, the explicit formula for the action of $\Tb_a$ can be given. To do this, we follow \eqref{operator}, setting $v=\k_x$ in \eqref{Distr}:
 \begin{equation}\label{OperDistr}
    (\Tb_a u)(x)=\langle(\Tb_a u)(\cdot), \k_x(\cdot)\rangle=\Fb_a(u,\k_x)= (a,u\k_x(\cdot))\end{equation}
    (recall that the reproducing kernel $\k(x,y)$ is real, therefore we can omit the complex conjugation.)

Expression \eqref{OperDistr} may be not always convenient since it involves the reproducing kernel which is expressed via the Bessel functions. Therefore, we find an equivalent representation, using the Fourier images.

Let $\widehat{a}(\x), \, \x\in\R^d,$ be the Fourier transform of the distribution $a$. Consider a function $u\in \Hc$;  it admits the representation \eqref{GeneralForm} with a certain $\f\in L_2(\Sbb)$. Recall the representation~\eqref{Operator1}. This representation has been derived under the condition that $a(x)$ is a function. However, under proper conditions, this representation can be refit so to cover the case of $a$ being a distribution in certain class.

More concretely, for $\f\in L_2(\Sbb)$, we consider a distribution $\varPhi_{\f}=\f\otimes \d(|\x|-1)$.
In more detail, $\varPhi_{\f}$ is a distribution in $\Ec(\R^d)$ (i.e., with compact support), acting on a function $f(\x)\in \Ec(\R^d)$ as
    \begin{equation}\label{varphi action}
        (\varPhi_{\f}, f)=\int_{\Sbb}\f(\x) f(\x) dS(\x).
    \end{equation}
Having this in mind, we can rewrite the expression \eqref{Operator1} as
    \begin{equation}\label{Varphi}
        \breve{\Tb}_a\f=(\widehat{a}*\varPhi_{\f})|_{\Sbb}.
   \end{equation}

    The expression in \eqref{Varphi} is at the moment only formal. In fact, the convolution of the distributions, $\widehat{a}*\varPhi_{\f}$ is well defined, since one of factors has compact support. However, the restriction of the distribution $\widehat{a}*\varPhi_{\f}$ to the unit sphere $\Sbb$ is not necessarily defined. In this way, the definition of the operator $\breve{\Tb}_a$ is related with the classical restriction problem. We will present later some classes of symbols for which this restriction problem can be solved, so that the expression \eqref{Varphi} is well defined for $\f\in L_2(\Sbb)$ and determines a function in $L_2(\Sbb)$.

    Our way of reasoning will be the following. To consider  symbols $a$ in some class $\Xc'(\R^d)$ of distributions on $\R^d$, we find a subclass $\Yc'\subset\Xc'(\R^d)$, dense in $\Xc'(\R^d)$ in the topology of  $\Xc'(\R^d)$. We will choose $\Yc'$ in such a way that the sesquilinear  form $\Fb_{b}$ is well defined on $\Hc$ for $b\in\Yc'$. Therefore, for fixed $u,v\in\Hc$, the value of sesquilinear form $\Fb_{b}(u,v)$ is a linear functional of the element $b\in\Yc'$, $W(b)=\Fb_{b}(u,v)$. If it turns out that this functional is continuous in $\Yc'$, it can be extended to $\Xc'(\R^d)$ by continuity, and thus defines the value of the sesquilinear form $\Fb_a(u,v)$ for all $a\in \Xc'(\R^d)$. Otherwise, if we are not able to establish the continuity of $W(b)$ on $\Yc'$ in the $\Xc'(\R^d)$-topology, it may happen that this functional is closable in this topology, so we can extend $W(b)$ to some subspace in $\Xc'(\R^d)$ by means of the closure procedure.

\subsection{Zero operators}The easiest case of such considerations concerns distributions generating the zero operator. The reasoning to follow extends our discussion in Section \ref{Sub.Degeneration}.
\begin{proposition}\label{PropZero} Let $a$ be a distribution in $\Sc'(\R^d)$. Suppose that $\supp(\widehat{a})\bigcap\overline{\Bb(0,2)}=\varnothing$. Then the operator $\Tb_a$ is well defined by \eqref{Varphi} and $\Tb_a u=0$ for any $u\in\Hc$.
\end{proposition}
\begin{proof} It is sufficient to prove the equality for distributions  with compact support, since such distributions are dense in $\Sc'(\R^d)$ in the topology of $\Sc'(\R^d)$. Following  the general strategy above, we take $\Xc'(\R^d)$ being the space of distributions in $\Sc'(\R^d)$ subject to the support condition, and as $\Yc'$ the subset of compactly supported distribution is accepted.

So we consider $b\in\Ec'(\R^d)$, still with $\supp(\widehat{b})\bigcap\overline{\Bb(0,2)}=\varnothing$. This implies that $\widehat{b}$ is a smooth function.
We take two arbitrary functions $u,v\in\Hc$. These functions admit the representation $u=\Ib \f$, $\Ib \psi$, for some $\f,\psi\in L_2(\Sbb)$. Consider the value of the sesquilinear $\Fb_b(u,v)$. We have, as above,

    \begin{equation}\label{Zero1}
        \Fb_b(u,v)=(b, u\overline{v})=(b, \Ib \f \overline{\Ib \psi})=c_d^2\int_{\Sbb}\widehat{b}(\x-\y)\int_{\Sbb}\f(\y)\psi(\y)dS(\y)dS(\x).
    \end{equation}
    It is easy to see that in this double integral, when the variables $\x,\y$ belong to the unit sphere, their difference, $\x-\y$, lies in the ball $\overline{\Bb(0,2)}$, where the function $\widehat{b}$ vanishes. This means that the integral in \eqref{Zero1} equals zero, in particular that the above functional $W(b)$ is the zero functional on $\Yc'$. By continuity, $W$ extends to the zero functional on distributions $a\in\Xc'(\R^d)$. Due to  the arbitrariness of $u,v\in \Hc$, this implies that $\breve{\Tb}_a=0$, $\Tb_a=0$.
    \end{proof}

    \subsection{Symbols with compact support} The next class  to study consists of distributions with compact support.
    \begin{theorem}\label{CompSuppDistr} Let $a$ be a distribution with compact support, $a\in\Ec'(\R^d)$. Then the operator $\Tb_a$ is well defined by \eqref{Form1.dist} or \eqref{Varphi} and bounded in $\Hc$.
    \end{theorem}\label{Th.bounded.compsupp}
    \begin{proof} By the structure theorem for distributions with compact support (see, e.g, Theorem in Sect.4.4. in \cite{Gelf}, page 116-117 of the AP edition of 1968), the distribution $a\in\Ec'(\R^d)$ has finite (differential) order and admits, for some $N$,  a representation
    \begin{equation}\label{FinOrder}
        a=\sum_{|\a|\le N}D^{\a}g_{\a}(x),
    \end{equation}
    where $\{g_{\a}, \ |\a|\le N\},$ is a collection of continuous functions, whose support can be chosen arbitrarily close to the support of $a$, and $D$ denotes the derivative in the distributional sense. Denote by $\r_0$ the radius of the ball in $\R^d$, centered at the origin and covering the support of all $g_{\a}$. Let $\theta(\x)$ be an even smooth ($C^\infty$) function, $\x\in\R^d$, such that $\theta(\x)=1$ for $|\x|\le2$, and with support in a ball with some radius $R_0>2$.  Consider the distribution $a_\theta=a*\widehat{\theta}$, where $*$ denotes the usual convolution of distributions; it is well defined since $a$ has compact support, and $\widehat{\theta}$ is the Fourier transform of $\theta$. For a typical term in \eqref{FinOrder}, we have
    \begin{equation}\label{a_theta}
        (D^\a g_\a)*\widehat{\theta}(x)=(g_\a *(D^\a\widehat{\theta}))(x)=\int_{\R^d} g_\a(y) D^{\a}(\widehat{\theta}(x-y))dy.
    \end{equation}
    Now, since $\theta$ has compact support and is smooth, its Fourier transform $\widehat{\theta}(y)$ is a smooth function decaying at infinity faster than $(1+|y|)^{-k}$ for any $k$, together with all derivatives. Therefore, for the integral in \eqref{a_theta}, we have the estimate
    \begin{equation}\label{a_theta1}
        |(D^\a g_\a)*\widehat{\theta}(x)|\le C_{g_\a}\int_{|y|\le\r_0}(1+|x-y|)^{-k}dy\le C'_{g_\a,k}(1+|x|)^{-k},
    \end{equation}
    for any $k>0$.
    For $k\ge d+1$ this means that for all $\a, \ |\a|\le N,$ the function $(D^\a g_\a)*\widehat{\theta}$ decays fast in $\R^d$. Therefore, $a_\theta=a*\widehat{\theta}$ is a  fast decaying function. Returning to its Fourier transform, we see that $\widehat{a_\theta}$ is a smooth (in particular, continuous) function in $\R^d$, which enables us to represent $\breve{\Tb}_{a_\theta}\f$ as
    \begin{equation}\label{a_theta2}
        \breve{\Tb}_{a_\theta}\f(\x)=\int_{\Sbb}\widehat{{a}_\theta}(\x-\y)\f(y) dS(y).
    \end{equation}
    Now recall that $\widehat{{a}_\theta}$ coincides with $\widehat{a}$ in the ball $\overline{\Bb(0,2)}$, and therefore, by Proposition \ref{PropZero},  $\breve{\Tb}_{a_\theta}=\breve{\Tb}_a$, so in both terms in \eqref{a_theta2} we can replace ${a}_\theta $ by $a$.
    \end{proof}

    \section{Homogeneous and homogeneously dominated symbols}\label{Sect:Homog}

 \subsection{Domination} In this section we continue with the study of the sesquilinear forms \eqref{Form1}. We present several classes of symbols that guarantee the boundedness of such sesquilinear forms.

 We start with a simple statement dealing with the domination for the symbols.

 \begin{proposition}\label{comparison} Let $a(x)$, $b(x)$ be two symbols in $L_{1,\loc}(\R^d)$. We suppose that $b(x)\ge0$ almost everywhere in $\R^d$, while $|a(x)|\le b(x)$ for almost all $x\in\R^d$. If the operator $\Tb_b$ is bounded in $\Hc$ then $\Tb_a$ is bounded as well, moreover, $\|\Tb_a\|\le \|\Tb_b\|$.
 \end{proposition}
 \begin{proof} The statement follows automatically from the standard equality $\|T\|=\sup_{u\in\Hc}|\langle Tu,u\rangle|$ for any operator $T$ in a Hilbert space.
 \end{proof}
 \begin{remark}When applying comparison methods, similar to Proposition \ref{comparison}, one usually looses in sharpness, since one ignores possible cancelations of the contributions of $a$, having different sign. This is a general feature of such comparison approach to proving the boundedness of the sesquilinear form \eqref{Form1} in various spaces. In some cases, in similar situations, it turns out to be possible to take into account such cancelations, arriving thus to more sharp results, see, in particular, our Example 4.4. For  the Bergman space, such improved results have been obtained in \cite{TaVi1},\cite{Zorboska03}. For the same sesquilinear form in the Sobolev spaces, it turned out to be possible even to reach the necessary and sufficient condition for this boundedness - see \cite{Maz}. We will consider this kind of problems in some later publications.
 \end{remark}

 Our first application of Proposition \ref{comparison} concerns homogeneously dominated symbols.

 \begin{definition}Let $a(x)$ be an $L_{1,\loc}$ function in $\R^d$. We say that $a$ is a \emph{homogeneously dominated} (HD) function of degree $\l$ if
 \begin{equation}\label{HD}
    |a(x)|\le A|x|^{\l}
 \end{equation}
 for sufficiently large $|x|$, with some constant $C$.
 \end{definition}

 \begin{theorem}\label{HD-estimateTh} Let the symbol $a(x)$ be a HD function of degree $\l$, $\l<-1$. Then the sesquilinear form $\Fb_a$ is bounded in $\Hc$.
 \end{theorem}
\begin{proof}By Proposition \ref{comparison} and Theorem \ref{CompSuppDistr}, it is sufficient to prove the boundedness in question for the symbol $a(x)=A|x|^{\l}$. The Fourier transform of this symbol equals $\widehat{a}(\x)=CA|\x|^{-\l-d}$ (with some constant $C=C_{\l,d}$). Now the boundedness of the operator follows from Theorem \ref{Criter.bounded}. In fact, for $\l<-1$, the function $|\x|^{-\l-d}$ belongs to $L^1$ on any sphere $S_\y$, moreover, the integral $\int_{S_\y} |\x|^{-\l-d}dS(\x)$ is uniformly bounded for all $\x\in\Sbb$ by $AC'(\l,d)$, what is exactly the requirement of Theorem \ref{Criter.bounded}.
\end{proof}
The conditions of Theorem \ref{HD-estimateTh} can be somewhat relaxed.

We call the positive function $\vf(t)$ on the semiaxis $t\ge0$ an admissible radial gauge function  (an \emph{a.r.g.f.}) if the function $\pmb{\pmb{\vf}}(x)=\vf(|x|)$ belongs to $L_{1,\loc}(\R^d)$ and its Fourier transform $\widehat{\pmb{\pmb{\vf}}}(\x)$, $\x=(\x_1,\x'), \, \x'\in\R^{d-1},$ possesses the following property:
\begin{equation}\label{AdmRadGauge}
 \int_{|\x'|\le 2} |\widehat{\pmb{\pmb{\vf}}}(0,\x')|d\x'<\infty.
\end{equation}

\begin{theorem}\label{relaxedOnBoundedness}
 Let $\vf(t)$ be an a.r.g.f., and $|a(x)|\le C\vf(|x|)$ for sufficiently large $|x|.$ Then the Toeplitz operator $\Tb_a$ with symbol $a$ is bounded in $\Hc$.
\end{theorem}
\begin{proof}Note first that, due to the spherical symmetry, the condition \eqref{AdmRadGauge} implies that, for any $\z^0\in\Sbb$, the integral of $|\widehat{\pmb{\pmb{\vf}}}|$ over any hyperdisk $(\x\cdot\z^0)=0$, $|\x|\le2$, is finite, moreover, it does not depend on the choice of $\z^0$. Now it remains to notice that the integral $\int_{S_{\z^0}}|\widehat{a}(\x)|dS(\x)$, as  in Theorem \ref{Criter.bounded}, can be majorated by the integral in \eqref{AdmRadGauge}.
\end{proof}

One can give examples of  admissible radial gauge functions. First, of course, these are the homogeneous functions as in Theorem \ref{HD-estimateTh}, i.e., $\vf(t)=Ct^{\l}$, $-d<\l<-1$. More examples are obtained by considering $\vf(t)$ being $t^{-1}$ with some logarithmic factor, e.g.,
$t^{-1}\log(2+t)^{-1-\e}$, $t^{-1}(\log(2+t))^{-1}(\log(2+\log(2+t)))^{-1-\e}$ etc. An important shortcoming for these and many other examples lies in the fact that the corresponding Toeplitz operators are not only bounded in $\Hc$ but even compact. We devote to this compactness phenomenon the next subsection.

\subsection{Compactness}\label{compact}
To study the conditions for the Toeplitz operator $\Tb_a$ generated by the sesquilinear form \eqref{Form1} to be compact in $\Hc$, we return to the case of the symbol $a$ having a compact support.
\begin{proposition}\label{Prop.comp}Let the symbol $a\in\Ec'(\R^d)$ have compact support. Then the Toeplitz operator $\Tb_a$ is compact. Moreover, it belongs to any Schatten class  $\SF_p, \, p>0.$
\end{proposition}
\begin{proof}As it is shown in the proof of Theorem \ref{CompSuppDistr}, the Fourier transform of the symbol is an infinitely  smooth function. Therefore, the operator $\breve{\Tb}_a$ is an integral operator with smooth kernel, which proves our statement. \end{proof}
Now  we can can describe a wide class of symbols generating compact operators.
\begin{theorem}Let $\vf(t)$ be an a.r.g.\,f. Suppose that $a(x)$ is a symbol such that $a(x)=o(\vf(|x|)), \, |x|\to\infty.$ Then the operator $\Tb_a$ in $\Hc$ is compact.
\end{theorem}
\begin{proof}
For a given $\ve>0$ it is possible to find $R_\ve$ so that $|a(x)|<\ve \vf(|x|)$ for $|x|>R_\ve$. Therefore we can represent the symbol $a$ as the sum, $a(x)=a_\ve(x)+a_\ve'(\ve)$, where $a_\ve(x)$ coincides with $a(x)$ for $|x|\le R_\ve$ and is zero otherwise. By our construction, the operator with symbol $a_\ve'(x)$ has norm not greater that $C\ve$, with some constant $C=C(\vf)$, while the operator with symbol $a_\ve(x)$ is compact due to Proposition \ref{Prop.comp}. Thus, the operator $\Tb_a$ can be norm approximated by compact operators and therefore is compact itself.
\end{proof}
A somewhat more complicated reasoning shows that for a rather wide  class of admissible radial gauge functions any symbol majorated by such functions generates a compact operator.

\section{Sesquilinear forms $\Fb_a(u,v) = \int_{\Sbb} a(\xi)(\Ib^*u)(\xi)\overline{(\Ib^*v)(\xi)}\,dS(\xi)$} \label{h-convolutions}
In this section we start developing the analysis of operators generated by sesquilinear forms of a different type.
For any $a=a(\xi) \in L_{\infty}(\Sbb)$, we define the following sesquilinear form on $\Hc$
\begin{equation*}
 \Fb_a(u,v) = \int_{\Sbb} a(\xi)(\Ib^*u)(\xi)\overline{(\Ib^*v)(\xi)}\,dS(\xi).
\end{equation*}
This form is obviously bounded, and thus defines the Toeplitz operator
\begin{equation} \label{eq:toep-L_2}
 (\Tb_{\Fb_a}u)(x) = \Fb_a(u,k_x(\cdot)).
\end{equation}
We have then
\begin{eqnarray*}
 (\Tb_{\Fb_a}u)(x) &=& \int_{\Sbb} a(\xi)(\Ib^*u)(\xi)\overline{(\Ib^*k_x)(\xi)}\,dS(\xi) \\
 &=& c_d \int_{\Sbb} a(\xi)(\Ib^*u)(\xi) e^{ix\xi}\,dS(\xi) = (\Ib a \Ib^*u)(x).
\end{eqnarray*}
The above representation leads directly to the following theorem.
\begin{theorem}
The $C^*$-algebra $\mathcal{T}(L_{\infty}(\Sbb))$ generated by all Toeplitz operators of the form \eqref{eq:toep-L_2} is commutative, coincides with set of all its generators $\Tb_{\Fb_a}$, $a=a(\xi) \in L_{\infty}(\Sbb)$, and is isomorphic and isometric to $L_{\infty}(\Sbb)$. This isomorphism is given by the following mapping
\begin{equation*}
 \Tb_{\Fb_a} \ \ \longmapsto \ \ a=a(\xi) \in L_{\infty}(\Sbb).
\end{equation*}
The Toeplitz operators $\Tb_{\Fb_a}$ possess the following properties:
\begin{itemize}
 \item [-] \ $\|\Tb_{\Fb_a}\| = \|a\|_{L_{\infty}(\Sbb)}$,
 \item [-]  \ $\mathrm{spec}\, \Tb_{\Fb_a} = \mathrm{closure}\,(\mathrm{ess\text{-}range}\,(a))$,
 \item [-]   each common invariant subspace for the operators $\Tb_{\Fb_a}$ is of the form \ $\Ib( L_2(M))$, where $M$ is a measurable subset of $\Sbb$ having positive measure.
\end{itemize}
\end{theorem}

To get an integral form of the Toeplitz operator $\Tb_{\Fb_a}$ we consider
\begin{eqnarray*}
 (\Tb_{\Fb_a}u)(x) &=& (\Ib a \Ib^*u)(x) = \Ib \left( a(\xi)\, c_d \lim_{R \to \infty} R^{-1} \int_{|y| < R} u(y) e^{-iy\xi}\, dy\right) \\
&=& c_d^2 \int_{\Sbb} a(\xi) \left(\lim_{R \to \infty} R^{-1}\int_{|y| < R} u(y) e^{-iy\xi}\, dy\right) e^{ix\xi}\,d\xi \\
&=& c_d \lim_{R \to \infty} R^{-1}\int_{|y| < R} \left( c_d \int_{\Sbb} a(\xi) e^{i(x-y)\xi}\, d\xi \right) u(y) dy \\
&=& c_d \lim_{R \to \infty} R^{-1}\int_{|y| < R} K(x-y)u(y) dy,
\end{eqnarray*}
where
\begin{equation*}
 K(x) = c_d \int_{\Sbb} a(\xi) e^{ix\xi}\, d\xi = (\Ib a)(x).
\end{equation*}

Given two functions $u$ and $v$ from $\Hc$, we define their \emph{$h$-convolution} (Herglotz convolution) as
\begin{equation*}
 (u \ast_h v)(x) = c_d  \lim_{R \to \infty} R^{-1}\int_{|y| < R} u(x-y)v(y) dy.
\end{equation*}
In this notation,
\begin{equation*}
 (\Tb_{\Fb_a}u)(x) = (K \ast_h u)(x), \quad \text{where} \ \ K(y) = (\Ib a)(y).
\end{equation*}

We list now some fundamental properties of $h$-convolutions.

\begin{proposition}
 Let $u$ and $v$ be two functions from $\Hc$, one of them being in $\Hc_{\infty} = \Ib(L_{\infty}(\Sbb))$. Then
\begin{itemize}
 \item [(i)] \quad $(u \ast_h v)(x) \in \Hc$,
 \item [(ii)] \quad $(u \ast_h v)(x) = (v \ast_h u)(x)$,
 \item [(ii)] \quad $\Ib^*(u \ast_h v) = \Ib^*(u) \cdot \Ib^*(v)$.
\end{itemize}
\end{proposition}

\begin{proof}
Let $\phi,\, \psi \in  L_2(\Sbb)$ be such that $u = \Ib(\phi)$ and $v = \Ib(\psi)$, or $\phi = \Ib^*(u)$ and $\psi = \Ib^*(v)$. Moreover, assume that one of them, either $\phi$ or $\psi$, belongs to $L_{\infty}(\Sbb)$. Then,
\begin{eqnarray*}
 (u \ast_h v)(x) &=& c_d  \lim_{R \to \infty} R^{-1}\int_{|y| < R} u(x-y)v(y) dy \\
 &=& c_d \lim_{R \to \infty} R^{-1}\int_{|y| < R} \left( c_d \int_{\Sbb} \phi(\xi) e^{i(x-y)\xi}\, d\xi \right) v(y) dy \\
&=& c_d \int_{\Sbb} \phi(\xi) \left( c_d\lim_{R \to \infty} R^{-1}\int_{|y| < R} v(y) e^{-iy\xi}\, dy\right) e^{ix\xi}\,d\xi \\
 &=& c_d \int_{\Sbb} \phi(\xi) \psi(\xi)\, e^{ix\xi}\,d\xi = \Ib(\phi \cdot \psi).
\end{eqnarray*}
But $\phi \cdot \psi \in L_2(\Sbb)$, thus $u \ast_h v = \Ib(\phi \cdot \psi) \in \Hc$.

Then, $u \ast_h v = \Ib(\phi \cdot \psi) = \Ib(\psi \cdot \phi) = v \ast_h u$.

And finally, $\Ib^*(u \ast_h v) = \Ib^*(\Ib(\phi \cdot \psi)) = \phi \cdot \psi =
\Ib^*(u) \cdot \Ib^*(v)$.
\end{proof}

\appendix
\section{Non-degeneracy in the Bergman and  Fock spaces}\label{SectDegenFock}
It seems to be useful to compare the degeneration expressed in Proposition \ref{PropNDegen} with the state of the similar property for the Bergman and the Fock spaces. As usual, $\Ac^2(\mathbb{B}^d)$ denotes the Bergman space  consisting of analytic functions  in the unit ball $\Bb(0,1)\equiv\mathbb{B}^d\subset\C^d$. In this space we consider operators generated by the sesquilinear form \eqref{Form1} with symbol $a\in L^1(\mathbb{B}^d)$. Suppose that $\Fb_a(u,v)=0$ for all $u,v\in \Ac^2(\mathbb{B}^d) $, in particular the sesquilinear form annuls  on all monomials, $\Fb_a(z^a,z^\b)=\int a(z)z^\a \overline{z}^\b dV(z)=0$. Consider the algebra $\mathfrak{A}$ consisting of finite linear combinations of functions $z^\a \bar{z}^\b$, where $\a,\b$ are multi-indices in $(\Z_+)^d$. This algebra consists, actually, of all polynomials of $2d$ real variables $x,y$. By the  Weierstrass approximation theorem, the closure of $\mathfrak{A}$ in $C(\overline{\mathbb{B}^d})$-norm
coincides with $C(\overline{\mathbb{B}^d})$. Since $\int_{\mathbb{B}^d}a fdV=0$ on all functions $f\in \mathfrak{A}$, by continuity, we have $\int_{\mathbb{B}^d}a fdV=0$ for any continuous in $\mathbb{B}^d$ function $f$, and this is possible only for $a=0$ (almost everywhere.) Further on, we will refer to the above kind of results as 'non-degeneracy' property.

Without essential changes, the above proof remains valid for  Toeplitz operators in the Bergman space in the polydisk.

Further on, a similar reasoning is applicable to the Bergman space of harmonic functions in the (real) ball $\DD_m$ in $\R^m$. For an even dimension $m=2d$, the non-degeneracy property follows from the same property for the analytic Bergman spaces, since the space of analytical functions is a subspace in the harmonic Bergman space. In an odd dimension, the result follows via a simple dimension reduction procedure. For a fixed $x_0\in\R^d\setminus 0$, we consider the subspace in the harmonic Bergman space, consisting of functions that are constant along the direction of $x_0$. The Weierstrass reasoning as above leads to the fact that the integral of $a$ along any straight line parallel to $x_0$ vanishes. Due to the arbitrariness of $x_0$, by the uniqueness property for the Radon transform, this implies that $a=0$.

A somewhat more complicated is the non-degeneracy situation for the Toeplitz operators in the Fock type space. We remind that the space $\Fs^2=\Fs^2(\C^d)$ consists of entire analytical functions in $\C^d$ that are square integrable with respect to the Gaussian measure $d\m(z)=\pi^{-d}dV(z)$, where $dV(z)$ is the Lebesgue measure in $\C^d$ identified in the usual way with $\R^{2d}$:
 \begin{equation*}
    \int_{\C^d}|u(z)|^2e^{-|z|^2}d\m(z)<\infty,
 \end{equation*}

 It is more convenient to consider in $\Fs^2$ the sesquilinear forms of the type  \eqref{Form1}, with the weight factor  explicitly separated, so the Toeplitz operator $\Tb_a$ in question is associated with the sesquilinear form
 \begin{equation}\label{FockForm}
 \Fb_a(u,v)=\int_{\C^d} a(z)u(z)\bar{v}(z)d\m(z)= \pi^{-d}\int_{\C^d}a(z)u(z)\bar{v}(z)e^{-|z|^2}dV(z).
 \end{equation}
 The representation \eqref{FockForm} of the sesquilinear form  is convenient since it can be expressed in the terms of the enveloping space $H=L_2(\C^d, d\m)$: $\Fb_a(u,v)=\langle au,v\rangle_H$.
 For the case of a symbol $a(z)$ having a compact support, the non-degeneracy property is proved identically with the Bergman space case. However, with the compactness of support condition  dropped, the situation becomes more complicated.
 In the paper \cite{GruVas} an example  of a \emph{radial} symbol $a(z)$, i.e., the one  depending on $r=|z|$ only, was constructed, so that $\Fb_a(u,v)=0$ for all pairs of functions $u,v\in \Fc^2(\C^d)$, for which the integral in \eqref{FockForm} converges. This symbol $a(z)$ \emph{grows} rather rapidly at infinity in such a way  that
 \begin{equation}\label{growing symbol}
    \limsup_{|z|\to\infty}|a(z)|e^{-|z|^2+c|z|^{\s}}>0
 \end{equation}
 for some  positive $c$ and some $\s\in(0,1)$. In particular, this means that the operator $\Tb_a$, defined initially on polynomials, has only the zero function as its value, and therefore extends to the zero operator by continuity. A more explicit example of the symbol $a$, again a radial one, has been constructed in \cite{BauLe}.

 It is interesting to understand the driving force in the construction of such 'degeneracy examples', in comparison to the degeneracy phenomenon in the Helmholtz space, depicted above, in Sect. 5.2.

 Let $a$ be  a function such that
 \begin{equation*}
 \Fb_a(u,v)=\pi^{-d}\int_{\C^d}a(z)u(z)\bar{v}(z)e^{-|z|^2}dV(z)=0
 \end{equation*}
 for all functions $u,v$ in the basis of $\Fc^2$, in other words,

 \begin{equation}\label{degen2}
    \int_{\C^d} a(z)z^\a \bar{z}^\b e^{-|z|^2}dV(z)=0
 \end{equation}
 for all multi-indices $\a,\b.$
 We pass to the ($2d$-dimensional real) Fourier transform, $$\Psi(\xi)=(2\pi)^{-d}\int_{C^d}a(x+iy)e^{-(x^2+y^2)} e^{i(x,y)\cdot \x}dxdy, \, \x\in\R^{2d},$$
 and we obtain from \eqref{degen2} the identity
 \begin{equation}\label{degen3}
    (\partial^\a\bar{\partial}^\b\Psi)(0)=0,
 \end{equation}
 for all multi-indices $\a,\b\in (\Z_+)^d$. So, the function $\Psi(\xi)$, the Fourier transform of $a(z)e^{-|z|^2}$, must have vanishing derivatives of all orders at the origin. For the (more computable) case of a radial symbol $a$, the above condition can be expressed (as it was done in \cite{GruVas}) in the terms of the one-dimensional Fourier transform. One should construct a smooth function $h(\r)$ that has zero of infinite order at zero, and take as $a(r)e^{-r^2}$ the inverse Fourier transform of $h(\r)$. Considerable ingenuity was needed to guess such function $h(\r)$ for which the inverse Fourier transform has a controlled growth at infinity, a slower one than the inverse Gaussian function, i.e.,
 \begin{equation}\label{degen4}
    a(r)e^{-r^2}=O(e^{-cr^{\s}}),  \, \s\in(0,1), c>0,
 \end{equation}
 so that the sesquilinear form be defined at least on all polynomials.

 A natural question arises about the sharpness of the above condition, i.e.,  whether there exist symbols with degeneracy property  but with just slightly slower growth than in \eqref{degen4}, for example,
 \begin{equation}\label{degen5}
    a(z)e^{-|z|^2}=O(e^{-cr}), c>0, \, r\to\infty.
 \end{equation}
 We will show now that such symbols do not exist.
 \begin{proposition}\label{PropDegen} If a nonzero symbol $a\in L^1_{loc}(\C^d)$ satisfies \eqref{degen5}, then the Toeplitz operator $\Tb_a$ in the Fock space is not a zero operator, moreover, there exist multi-indices $\a$, $\b$ such that
 \begin{equation}\label{degen6}
    \Fb_a(z^\a, \bar{z}^\b)\ne0.
 \end{equation}
 \end{proposition}

 \begin{proof}
 Suppose, on the contrary, that $\Fb(z^\a, \bar{z}^\b)=0$ for all multi-indices $\a,\b$. This means that $\Fb_a(p(z), q(\bar{z}))=0$ for all analytical  polynomials $p,q$, and, therefore,
 \begin{equation}\label{degen7}
    \int_{\C^d} a(x+iy)e^{-|z|^2}H(x,y) dxdy=0
 \end{equation}
 for any polynomial $H$ of real variables $x,y$. Fix some positive $c'<c_a$. By the Bernstein weighted approximation theorem (see, e.g., \cite{Koosis} for the multi-dimensional version), any continuous compactly supported function $G(x,y)$ can be approximated on $\R^d$ by polynomials of $x,y$ variables uniformly on $\R^d$ with weight $e^{-c'|z|}$. Using this approximation in \eqref{degen7}, we see that $\int_{\C^d} a(x+iy)e^{-|z|^2}\Gc(x,y) dxdy=0$ for any compactly supported function $G$, which implies that $a=0$.
 \end{proof}
 \begin{remark} The condition \eqref{degen5} can be slightly improved, using some sharper form of the Bernstein approximation theorem, see, e.g., \cite{Lorentz}. One can assume that $|a(z)|e^{-|z|^2}\le \g(|z|)$ for large $|z|$, where $\g$ is a sufficiently regular  non-increasing function on the positive semi-axis such that $\frac{\log \g(t)}{1+t^2}$ does not  belong to $L^1(0,\infty)$.
 \end{remark}

 One can also understand the non-degeneracy described in Proposition \ref{PropDegen} using the properties of the Fourier transform of $a(z)e^{-|z|^2}$. As  explained above, the degeneracy occurs as soon as this Fourier transform  has zero of infinite order at the origin. For a nonzero function to have zero of infinite order,  it is necessary to belong to some
non quasi-analytical class. This latter  property for the Fourier transform of $a$ is related to the decay rate of $a(z)e^{-|z|^2}$, see, e.g., \cite{PW}.


\begin{thebibliography}{333}

\bibitem{Agmon1}Agmon, S. \emph{A representation theorem for solutions of the Helmholtz equation and resolvent estimates for the Laplacian.} Analysis, et cetera. Research Papers Published in Honor of Jurgen Moser's 60th Birthday, Academic Press 1990, 39--76.

\bibitem{Agmon2} Agmon, S. \emph{Representation theorems for solutions of the Helmholtz equation on $\R^n$.} Differential operators and spectral theory, Amer. Math. Soc. Transl. Ser. 2, 189, Amer. Math. Soc., Providence, RI, 1999, 27-–43.

\bibitem{AgmonHorm} Agmon, S.;  H\"ormander, L., \emph{Asymptotic properties of solutions of differential equations with simple characteristics.} J. Analyse Math. \textbf{30} (1976), 1--38.

\bibitem{AlexRoz}Alexandrov, A.; Rozenblum G. \emph{Finite rank Toeplitz operators: Some extensions of D. Luecking's theorem.} J. Funct. Anal., 256 (2009) 2291--2303.

\bibitem{AlvarezFS}Alvarez, J.; Folch-Gabayet, M.; P\'erez-Esteva, S. \emph{Banach spaces of solutions of the Helmholtz equation in the plane.} J. Fourier Anal. Appl. \textbf{7} (1) (2001) 49--62.

\bibitem{Atkin-Han} Atkinson, K.; Han, W., Spherical Harmonics and Approximations
on the Unit Sphere: An Introduction. Lecture Notes in Mathematics, v. 2044, Springer-Verlag, Berlin Heidelberg, 2012.

\bibitem{BarceloFPR}Barcelo, J.; Folch-Gabayet, M.; P\'erez-Esteva, S.; Ruiz, A. \emph{Toeplitz operators on Herglotz wave functions.} J. Math.Anal.Appl., \textbf{358} (2009) 364--379.

\bibitem{BauHerVas} Bauer, W.; Herrera Ya\~nez, C.;  Vasilevski, N. \emph{Eigenvalue characterization of radial operators on weighted Bergman spaces over the unit ball.} Integral Equations Operator Theory, \textbf{78}, (2) (2014) 271--300.

\bibitem{BauLe}Bauer, W.; Le, T. \emph{Algebraic properties and the finite rank problem for Toeplitz operators on the Segal-Bargmann space. } J. Funct. Anal., \textbf{261}, (9) (2011) 2617--2640.

\bibitem{Colton}Colton D.; Kress R. Inverse Acoustic and Electromagnetic Scattering Theory, 2nd Edition, Springer, Appl. Math. Sciences, \textbf{93}, 1998.

\bibitem{Gelf}Gelfand, I.M.; Shilov, G.E.  Generalized functions, II. Spaces of fundamental and generalized functions, Academic Press, 1968.

\bibitem{Grad} Gradshtein I.S.; Ryzhik I.M., Tables of Integrals, Series and Products. Acad. Press, 2007.

\bibitem{GruVas}Grudsky, S.; Vasilevski, N., \emph{Toeplitz operators on the Fock space: radial component effect, } Integral  Equations Operator Theory, \textbf{44} no. 1, (2002), 10--37.

\bibitem{Guo} Guo, K., \emph{A uniform $L^p$ estimate of Bessel functions and distributions supported on $S^{n-1}$.} Proc. AMS, \textbf{125} (5) (1997), 1329--1340.


\bibitem{Helgason81} Helgason, S., Topics in Harmonic Analysis on Homogeneous Spaces. Birkh\"auser, Boston, 1981.

\bibitem{Koosis} Koosis P. The Logarithmic Integral. I. Cambridge University Press. Cambridge, 1988.

\bibitem{Lorentz}Lorentz, G.; von Golitschek, M.;  Makovoz, Y. Constructive Approximation. Advanced Problems. Grundlehren d. Mathematischen Wissenschaften, Springer, 1996.

\bibitem{Lue} Luecking, D. \emph{Finite rank Toeplitz operators on the Bergman space}, Proc. Amer. Math. Soc. 136 (2008), no. 5, 1717-–1723.

\bibitem{Maz} Maz'ya, V.; Verbitsky, I.  \emph{The Schrodinger operator on the energy space: boundedness and compactness criteria.} Acta Math. 188 (2002), no. 2, 263--302.



\bibitem{PW}Paley, R.; Wiener N. Fourier transforms in the complex domain. AMS, 1934.



\bibitem{Perez.Herglotz}   P\'erez-Esteva; S, Valenzuela-Diaz, S.,\emph{ Reproducing kernel for the Herglotz functions in $\R^n$ and solutions of the Helmholtz equation.} To appear.

\bibitem{RodTh}Rodberg, L.; Thaler, R., Introduction to the Quantum Theory of Scattering, Academic Press, 1965.

\bibitem{RozShir} Rozenblum, G.; Shirokov, N. \emph{Some weighted estimates for the $\bar{\partial}$-equation and a finite rank theorem for Toeplitz operators in the Fock space.} Proc. Lond. Math. Soc. (3) 109 (2014), no. 5, 1281--1303.

\bibitem{RozTObzor}Rozenblum, G., \emph{Finite rank Toeplitz operators in Bergman spaces}, in: Around the Research of Vladimir Maz'ya. III: Analysis and Applications.", Springer 2010, 331--358,  \verb"arXiv:0904.0171"

\bibitem{RozVas1} Rozenblum, G.; Vasilevski, N. \emph{Toeplitz operators defined by sesquilinear forms: Fock space case}. J. Funct. Anal. \textbf{267} no.11, (2014) 4399--4430.

\bibitem{RozVas2}Rozenblum, G.; Vasilevski, N. \emph{Toeplitz operators defined by sesquilinear forms: Bergman space case.} J. Math. Sci. (Springer) \textbf{213} no.4, (2016) 582--609.


\bibitem{TaVi1} J.~Taskinen; J.~Virtanen,  \emph{Toeplitz operators on Bergman spaces with locally integrable symbols.} Rev. Mat. Iberoam. \textbf{26} (2010), no. 2, 693--706.

\bibitem{Vasil}Vasilevski, N., Commutative Algebras of Toeplitz Operators on the Bergman Space. Birkh\"auser, 2008.



\bibitem{Yafa}Yafaev, D., Scattering Theory; Some Old and New Problems. Lecture Notes in Math., 1735, Springer, 2000.




\bibitem{Zorboska03} N.~Zorboska, \emph{Toeplitz operators with BMO symbols and the Berezin transform.} Int. J.  Math. Sci. 2003, no. 46, 2929-2945.

\end{thebibliography}
\end{document}